\documentclass[reqno, 11pt]{amsart}

\makeatletter
\@namedef{subjclassname@2020}{%
  \textup{2020} Mathematics Subject Classification}
\makeatother

\usepackage{ulem}

\usepackage{amsfonts, amsthm, amsmath, amssymb}
\usepackage{hyperref}
\hypersetup{colorlinks=false}

\usepackage{bbm}  

 \usepackage{tabu}

\usepackage[margin=1in]{geometry}

\usepackage{helvet}

\RequirePackage{mathrsfs} \let\mathcal\mathscr

\usepackage{hyperref}
\usepackage{dsfont}
\usepackage{upgreek}
\usepackage{mathabx,yfonts}
\usepackage{enumerate}

\numberwithin{equation}{section}

\newtheorem{theorem}{Theorem}[section] 
\newtheorem{lemma}[theorem]{Lemma}
\newtheorem{proposition}[theorem]{Proposition}
\newtheorem{corollary}[theorem]{Corollary}

\theoremstyle{definition}
\newtheorem{example}[theorem]{Example}
 \newtheorem*{acknowledgements}{Acknowledgements}
\newtheorem{remark}[theorem]{Remark}

\newtheorem{definition}[theorem]{Definition}

\renewcommand{\emph}[1]{\textit{#1}}

\renewcommand{\phi}{\varphi}

\newcommand{\PP}{\mathbb{P}}

\newcommand{\QQ}{\mathbb{Q}}

\renewcommand{\leq}{\leqslant}

\renewcommand{\geq}{\geqslant}

\renewcommand{\c}{\mathbf{c}}

\newcommand{\z}{\mathbf{z}}
 \renewcommand{\b}{\mathbf{b}}

\renewcommand{\r}{\mathbf{r}}

\let\emptyset\varnothing

\DeclareSymbolFont{bbold}{U}{bbold}{m}{n}
\DeclareSymbolFontAlphabet{\mathbbold}{bbold}

\newcommand{\md}[1]{  \left(\textnormal{mod}\ #1\right)}
 
\renewcommand{\P}{\mathbb{P}}
\newcommand{\Q}{\mathbb{Q}}
\newcommand{\F}{\mathbb{F}}
\newcommand{\N}{\mathbb{N}}
\newcommand{\R}{\mathbb{R}}

\newcommand{\Z}{\mathbb{Z}}
\renewcommand{\l}{\left}

\renewcommand{\r}{\right}
\renewcommand{\b}{\mathbf}
\renewcommand{\c}{\mathcal}
\renewcommand{\epsilon}{\varepsilon}

\renewcommand{\leq}{\leqslant}
\renewcommand{\geq}{\geqslant}
\renewcommand{\#}{\sharp}

\newcommand{\beq}[2]
{
\begin{equation}
\label{#1}
{#2}
\end{equation}
}

\title
[Diophantine stability and second order terms] 
{Diophantine stability and second order terms}

\author{Carlo Pagano} 
\address{Department of Mathematics\\ Concordia University
\\ H3G1M8, Montreal, Canada}
\email{carlo.pagano@concordia.ca}

\author{Efthymios Sofos} 
\address{Department of Mathematics\\
 Glasgow  University\\ G12~8QQ, Glasgow, United Kingdom}
\email{efthymios.sofos@glasgow.ac.uk}

\date{}

\begin{document} 
\begin{abstract}
We establish a Galois-theoretic trichotomy governing
Diophantine stability for genus $0$ curves. We use it to 
 prove that  the curve associated to the 
Hilbert symbol is Diophantine stable with probability $1$. Our asymptotic formula 
 for the second
order term    exhibits strong bias towards instability. 
\end{abstract}

\maketitle

\setcounter{tocdepth}{1}
\tableofcontents

\section{Introduction}   \label{s:intro} 
Mazur--Rubin put forward the ``minimalist philosophy" which states that 
a variety defined over $\mathbb Q$ should  typically be \textit{Diophantine stable}, i.e.
it should    acquire no new  points in finite extensions of $\Q$, see~\cite{MR3805014}. They proved many instances of this phenomenon 
for  elliptic curves and later revisited    this topic by
studying averages of modular symbols in~\cite{modular}. The present paper is inspired by their
statistical view-point. For example, we show that $100\%$ of smooth, projective, genus $0$ curves are Diophantine stable
and we give asymptotics for the error term.  It turns out that there is 
a second order term whose logarithmic exponent  has a Galois-theoretic interpretation.

Fix a finite number field extension $L/\Q$ once and for all for the rest of this paper.
We denote the Galois closure by $N(L)/\Q$  and  the corresponding Galois group by    
$\mathrm{Gal}(N(L)/\Q)$.
The set $X_L$ of roots of the minimal polynomial  of a primitive element of $L/\Q$ 
is a transitive $\mathrm{Gal}(N(L)/\Q)$-set. 
We define $$\delta_L:=\frac{\#\{g \in \mathrm{Gal}(N(L)/\Q): g \textrm{ has an orbit of odd length}\}}{\# \mathrm{Gal}(N(L)/\Q)},$$  where, as usual,  
 the elements $g\in \mathrm{Gal}(N(L)/\Q)$ can be  viewed as a permutation on $X_L$. It is clear that $\delta_L\neq 0$ and we shall see that $\delta_L=1$ 
is equivalent to the set $$ \c A_L=\left\{p \textrm{ finite prime in }\Q: \textrm{its decomposition group in } N(L)
 \textrm{ has only orbits of even size} \right\}$$ 
  consisting
 of all but finitely many primes in $\Q$.  
For a    prime $p$   let $\mu_p:= \mathrm{vol}(s,t\in \Z_p:(s,t)_{\Q_p}=1)$,
where $(s,t)_{k}$ denotes the Hilbert symbol in a local field $k$ and $ \mathrm{vol}$ is the normalised $p$-adic Haar measure.
For $s,t \in \Z^2$ we denote  the curve given by $s x_0^2+ t x_1^2=x_2^2 $ in $ \P^2 $ as $C_{s,t} $. When $L/\Q$ is a finite extension 
of number fields the curve $C_{s,t}$ is called Diophantine stable over $L$  if and only if $$C_{s,t}(L)= C_{s,t}(\Q).$$
\begin{theorem}
\label{thm:therm1}
Fix any finite number field extension $L\supsetneq \Q$, any constant $A>0$  and a vector $(a,b)\in \{1,-1\}^2 $ 
that does not equal $(-1,-1)$ if $L$ is not totally complex.
Then for $B\geq 3$ 
we have  
$$  \# \left\{(s,t)\in \Z^2:\hspace{-0.2cm}
\begin{array}{l}
-B\leq s,t \leq B , \\
\mathrm{sign}(s)=a, \mathrm{sign}(t)=b  , \\
C_{s,t} \ \mathrm{Diophantine \ stable \ over \ }L
\!\!\!
\end{array}\right\}=B^2-\frac{c_L B^2}{(\log B)^{\delta_L}}
\bigg( 1 
+O\left(\frac{1}{ (\log \log \log B)^A}\right)
\bigg)
,$$where $$c_L=\frac{1}{\Gamma(1-\delta_L/2)^2}
\l( 1+\mathds 1 (\delta_L=1) (a,b)_\R \prod_{ p\in \c A_L  } (2\mu_p-1) \r)
\prod_{\substack{p=2\\p \ \mathrm{ prime}}}^\infty  
 \begin{cases}\mu_p(1-1/p)^{-\delta_L},  & p\notin \c A_L, \\   (1-1/p)^{-\delta_L}, & p\in \c A_L,\end{cases}$$ $\Gamma$ denotes the Euler 
gamma function and the implied constant depends only on $A$ and $L$.\end{theorem} 

\begin{example}
\label{ex:gaussian} Using Weil restriction, 
we can reinterpret 
Theorem~\ref{thm:therm1}  in the context of~\cite[Conjecture 3.8]{conconcon}. 
For example, when $L=\mathbb Q(\sqrt{-1})$, we   write $x_i=y_i+z_i \sqrt{-1} $ with $y_i,z_i\in \Q$ so that 
 $C_{s,t}$ becomes  
$$
X: \hspace{1cm}   s (y_0^2-z_0^2)+t(y_1^2-z_1^2)=(y_2^2-z_2^2), \ \ sy_0 z_0+t y_1 z_1=y_2 z_2   \hspace{1cm}   \subset  \P^5\times \mathbb A^2
$$ that is equipped with the map $\pi:X\to \mathbb A^2$ sending $(y,z,s,t) $ to $(s,t)$. 
Note that the variety $X$ is not smooth as can be seen by considering the points $s=0,t=1$ and 
$y_0=z_0=1$ and all other $y_i,z_i$ to be $0$.
The fibres of $\pi$ give  a family of  intersections of   quadrics in $\P^5$ and the secondary term in 
 Theorem~\ref{thm:therm1} provides an asymptotic for the probability with which they 
have a $\Q$-rational point; the logarithmic exponent is $\delta_{\mathbb Q(\sqrt{-1})}=1/2$.
\end{example}

\subsection{Perfectly unstable}
As $C_{s,t}$ has  $0$  genus, it is obvious that 
 it is Diophantine unstable  
when $C_{s,t}(\Q)\neq \emptyset $. What is the proportion of obviously unstable curves inside all unstable ones?   
Let $$ r_L=\lim_{B\to\infty} \frac{\#\{(s,t)\in (\Z\cap [-B,B])^2:  C_{s,t} \textrm{\ Diophantine unstable over\ } L \}}
{\#\{(s,t)\in (\Z\cap [-B,B])^2:   C_{s,t}(\Q) \neq \emptyset \}}.$$
This limit always exists in $[1,\infty]$ as the numerator is asymptotic to $B^2/(\log B)^{\delta_L}$ by Theorem~\ref{thm:therm1}
and the denominator will be proved to be    asymptotic to $B^2/ \log B $ in Theorem~\ref{thm:analytic}.
We say that $L/\Q$ is \dotuline{perfectly unstable} when $r_L=\infty$, in other words, when the total number of   unstable curves over $L$ 
far exceeds the number of obviously unstable curves.

We next give a Galois-theoretic characterization of the ratio $r_L$:
\begin{theorem} [Trichotomy]\label{thm:perfectly0}\textnormal{(i)}  We have $r_L=1$ equivalently when $\#\c A_L\leq 1$. This is also 
equivalent to 
every $C_{s,t}$ having a point in $L$ if and only if it has a point in $\Q$.
 \\
\textnormal{(ii)}  We have $1<r_L<\infty  $ equivalently when $2\leq \#\c A_L<\infty  $. \\
\textnormal{(iii)}  $($\textrm{Perfectly unstable}$)$ We have $r_L=\infty$ equivalently when $\#\c A_L=\infty$.
\end{theorem} 
This is a simplified version of    Theorem~\ref{Thm: charact genus 0 stability} and other results in \S\ref{s:diophantnstabl},
that will  apply to any finite extension $L/K$ and genus $0$ curve.
Next, we  give  statements that make no reference to $\c A_L$.    Firstly,  by cohomology, 
every  extension  $L/\Q$ with odd degree $[L:\Q]$  has $r_L=1$ and is thus not perfectly unstable. 
\begin{theorem} \label{thm:perfectly} 
 \textnormal{(i)} A finite Galois extension  $L/\Q$ is perfectly unstable if and only if $[L:\Q]$ is even.
\\ \textnormal{(ii)} There are infinitely many    $L/\Q$ with $[L:\Q]=6$  that are not perfectly unstable and have $r_L\neq 1$. 
\end{theorem}

Finally, let us remark that 
the recent work   of   Fouvry--Koymans--Pagano~\cite{fouvrini} 
and Koymans--Morgan--Smit~\cite{koymorsmi}
dealt with a    situation analogous to 
Theorem~\ref{thm:therm1} 
using  a hyperbolic height.
These works provide a statistically valid  formula  for   Selmer groups of rational quadratic twists 
by square-free numbers $d$
over a fixed quadratic extension
 $L/\mathbb Q$:
in the language of the present paper, the dominant term of their formula consists precisely of the primes that  divide   
$d$ and 
lie  in $\c A_L$.

\subsection{The analytic result}

For any set of primes $\c P$ we define 
$$N(B,\c P):= \# \left\{(s,t)\in \Z^2: \begin{array}{l}
-B\leq s,t \leq B , \\
\mathrm{sign}(s)=a, \mathrm{sign}(t)=b  , \\
s x_0^2+ t x_1^2=1 \textrm{ has a }\Q_p\textrm{-point } \forall p\in \c P
\end{array}\right\}.$$  In~\S\ref{s:diophantnstabl}  we shall use class field theory to express the counting function in Theorem~\ref{thm:therm1} 
through  $ N(B,\c P_L)$.
We shall deduce Theorem~\ref{thm:therm1}    by the following result, proved in \S\ref{s:checkingsolubility}.

For a primitive Dirichlet character $\psi$ denote  its conductor by $q_\psi$. The symbol $(\frac{\cdot }{\cdot})$ denotes the Kronecker quadratic symbol.
We will deal with sets of primes  that are sufficiently random, in the 
sense that they are independent to the quadratic residues modulo every large enough discriminant.
\begin{definition}\label{def:indepe} We say that a set of primes $\c P$ is ``sufficiently random" if  
\begin{enumerate} \item for   all primitive Dirichlet characters $\psi$   and odd square-free integers $\beta$ coprime to $q_\psi$
  there exists $c_\psi(\beta) \in \mathbb C$ such that $$\lim_{x\to \infty} \frac{1}{x} \sum_{\substack{ p\leq x \\ p\in \c P }} 
\psi(p)\left(\frac{p}{\beta}\right) \log p = c_\psi(\beta) ,$$
\item  there exists an ascending unbounded  function  $\c L:[1,\infty)\to [1,\infty)$  
  such that  for each fixed $A>0$ and all    $\psi,\beta$ as above 
we have    \begin{equation}\label{AssumptionsonAAA} \ \ \ x\geq \max\{\c L(q_\psi),\exp(\beta^{1/200})\}
\Rightarrow  \sum_{\substack{ p\leq x \\ p\in \c P }} 
\psi(p)\left(\frac{p}{\beta}\right) \log p =
c_\psi(\beta) x+ O\l(  \frac{x}{(\log x)^A} \r),\end{equation}
where the  implied constant   depends only on $A$ and $\c P$,
\item  $\varpi:= c_1(1)$ is non-zero, \item       if $\varpi=1$ then $\c P$ contains all large enough primes,
\item for all but finitely many odd square-free $\beta$ we have    $\sup_\psi\{|c_\psi(\beta)|\}= 0 $, where the supremum is taken over all primitive 
Dirichlet characters $\psi$ with conductor coprime to $\beta$.
\end{enumerate}  \end{definition} The set of primes on any 
 coprime arithmetic progression is an example of such a set $\c P$. In fact,
any set of primes coming from Chebotarev conditions is ``sufficiently random".

If  $\varpi\neq 1$ we let $z_B$ be the largest solution of $\log B \geq (\log \c L(\mathrm e^{3z} ))^{\frac{8}{1-\varpi}}$. 
If $\varpi=1$ we let $z_B=\log \log B$. 
\begin{theorem}\label{thm:analytic} Assume that $\c P$ is any ``sufficiently random"
set of primes in the sense of Definition~\ref{def:indepe}. Fix any $(a,b) \in \{-1,1\}^2 $ and $A>0$.
Then for all $ B\geq \c L(\mathrm e^{3z_B})$   we have  
$$N(B,\c P)= c(\c P)\frac{B^2}{(\log B)^{\varpi} }+O\left(\frac{B^2}{(\log B)^{\varpi}(\log \min\{\log \log B, z_B\})^{A}}\right)
,$$ where    the  implied constant depends at most on $A$ and $\c P$ and $$c(\c P)=\frac{1}{\Gamma(1-\varpi/2)^2}
\l( 1+\mathds 1 (\varpi=1) (a,b)_\R
\prod_{ p\notin \c P  } (2\mu_p-1) \r)
\prod_{\substack{p=2\\p \ \mathrm{ prime}}}^\infty  
 \begin{cases}\mu_p(1-1/p)^{-\varpi},  & p\in \c P, \\   (1-1/p)^{-\varpi}, & p\notin \c P.\end{cases}$$ 
\end{theorem}  
When $\c P$ consists of all primes Serre  proved the  upper bound $O(B^2/\log B)$ using the large sieve~\cite[Th\'eor\`em 2]{MR1075658}
and asked whether this is the right order of magnitude~\cite[Exemple 2]{MR1075658}.
This was answered in the affirmative  by Friedlander--Iwaniec~\cite{MR2675875}, who used 
the large sieve inequality for quadratic characters to prove   asymptotics. We follow closely their approach as
 the main   changes needed concern   sums  of the form $$ \sum_{\substack{ n\leq x \\
p\mid n \Rightarrow p\in \c P }} a_n \chi(n),$$ where $\chi$ is a non-principal Dirichlet character. 
It is here where the third assumption regarding $\c P$ in Theorem~\ref{thm:analytic} is needed.
Theorem~\ref{thm:analytic}  recovers the Serre case because when $\c P$ contains all primes, the three assumptions   
are  satisfied with $\c L(q)=\mathrm e^q$ by the Siegel--Walfisz theorem for primes in progressions.

Friedlander--Iwaniec~\cite[Theorem 4]{MR2675875}  
also proved   matching upper and lower bounds for any set   of primes  $\c P$ 
of density $\leq 1/2$ via  the Brun sieve. Their result does not make any assumption on the structure of $\c P$.
One may interpret Theorem~\ref{thm:analytic} as saying that 
if  $\c P$ has the specific structure in Definition~\ref{def:indepe}  then 
their method yields asymptotics regardless of density. Tim Browning asked us         
whether there is a set of primes $\c P$ and a positive constant $c$ such that both limits 
 $$  \liminf_{B\to \infty} \frac{N(B,\c P)}{B^2(\log B)^{-c}}, \ \ \ 
\limsup_{B\to \infty} \frac{N(B,\c P)}{B^2(\log B)^{-c}} $$ exist but are not equal; we do not have an answer.

\subsection{The geometric-large sieve}

The secondary goal of this paper is to show that it is possible to adapt the $W$-trick of Green--Tao
in order to prove   asymptotics for the number of everywhere locally soluble varieties in some generality.
A crucial element of this   strategy is a common generalisation of the geometric and the large sieve that we shall   give in Theorem~\ref{thm:generalis}.
Let us first see what does the combined  sieve say about the problem of Loughran-Smeets~\cite{MR3568035},
where the random equations are given by the fibres of a  dominant morphism $f: V \to \PP^n$. Here,  $V$ is a proper,  smooth projective variety over $\Q$,  
$f$ has a   geometrically integral generic fibre and   the fibre  over every
codimension one point of $\P^n_\Q$ has an irreducible component of multiplicity one.

At this level of generality Loughran--Smeets~\cite[Theorem 1.2]{MR3568035}
showed that when we  order   $\P^n(\Q)$ by   the standard Weil height $H$ on $\PP^n(\QQ)$
and assume that    at least one fibre of $f$ is everywhere locally solvable,
then there exists a Galois-theoretic constant $\Delta(f)$ (given in~\cite[Theorem 1.2]{MR3568035})
such that   the number of everywhere locally soluble fibres satisfies
\beq{eq:uplougsmets}{\#\l\{x \in \P^n(\Q):x\in f\l(V\l(\mathbb A_\Q\r)\r), H(x)\leq B \r\}=O\l( \frac{B^{n+1} } {(\log B)^{\Delta(f)}}\r).}
They conjectured that   this is the right order of magnitude~\cite[Conjecture 1.6]{MR3568035}. We prove that 
certain fibres can be ignored when trying to verify this conjecture.  \begin{theorem}\label{thm:loughrnsmt}
Let $f, V$ and $H$ be as above. Let  $Y$ be any closed subscheme of $\P^n_\Q$ of codimension $k\geq 2 $.
For all $B,z\geq 2 $ we  have  \[ \#\l\{\b x \in \P^n(\Q): \begin{array}{l}
x\in f\l(V\l(\mathbb A_\Q\r)\r) , H(x)\leq B, \\ x \md p \in Y(\F_p) \textrm{ for some prime } p>z
  \end{array}\hspace{-0,15cm} \right\} \ll \frac{1}{\min\{z^{k-1}(\log z) ,B^{1/5}\}} \frac{B^{n+1}}{(\log B)^{\Delta(f)}}
 ,\] where the implied constant depends only on $f, V$ and $Y$. 
\end{theorem} The prefactor $1/\min$ gives a saving over the conjectured growth if   $z=z(B)$   tends to infinity, no matter how slow. 
This   allows us, for example, to ignore   varieties whose coefficients have a   common prime divisor $p>z$. That 
   enables  one  to write easy detector functions for $p$-adic solubility. 

To state the geometric-large sieve let us recall the two individual sieves. The geometric sieve is one of the few sieves   that  
are capable of proving     asymptotics; it was introduced by Ekedahl~\cite{MR1104780}.
Its   effective version is due to Bhargava~\cite[Theorem 3.3]{1402.0031}:
\begin{lemma}[Geometric sieve] \label{lembharg} 
Let $U$ be a compact region in $\R^n$ having finite measure, and let $Y$ be any closed subscheme of $\mathbb A_{\Z}^n$ of codimension $k\geq 1 $.
Let $B,z$   be   real numbers $\geq 2 $. Then   \[\#\{\b x \in B U\cap \Z^n: \b x \md p \in Y(\F_p) \textrm{ for some prime } p>z \}=
O\l(\frac{B^n}{z^{k-1}\log z }+B^{n-k+1}\r) ,\] where the implied constant depends only on $U$ and $Y$.\end{lemma} 
It has been   applied to  problems of positive density, e.g. square-free values of polynomials (Poonen~\cite{MR1980998})
or  solubility of families of Diophantine equations in many variables (Poonen--Voloch~\cite{MR2029869}).
 
Linnik's large sieve~\cite[\S 4]{MR2426239} gives upper bounds in  problems     of zero density.
Let us recall the  higher-dimensional  version   by Serre~\cite{MR1075658}:\begin{lemma}[Large sieve] \label{lemser}
Let  $\Omega$ be a subset of $\mathbb \Z^n$ that is contained in a cube of side $B\geq 1 $.
Let $m$ be a strictly positive integer. For a prime $\ell $ define  $\omega(\ell)$ by  $\#\Omega(\Z/\ell^m \Z)=\ell^{nm}(1-\omega(\ell) ) $. Then 
$$ \#\Omega \leq (2B)^n/L(B^{1/{2m}}),$$ where $L(z)=\sum_{q\leq z } \mu(q)^2 \prod_{\ell\mid q } \omega(\ell)/(1-\omega(\ell) )$.
\end{lemma} We can now give the common generalisation of the two sieves: \begin{theorem}[The geometric-large sieve]
\label{thm:generalis} Keep the setting of   Lemmas~\ref{lembharg}-\ref{lemser} and assume that 
  $\limsup\limits_{p\to\infty} \omega(p) \neq 1 $. Then for all $B, z \geq 2 $ one has 
  \[\#\{\b x \in B U\cap \Omega: \b x \md p \in Y(\F_p) \textrm{ for some } p>z \}
\!=\!O\!\l(\frac{1}{z^{k-1} (\log z) }\frac{B^n}{L(B^{\frac{1}{4m}})}
+ B^{n-\frac{(k-1)}{4m} }  \r), \] where  the implied constant depends only on $m,U, Y$ and $\limsup \omega(p)$.
\end{theorem} This gives a saving compared to the bounds that any of the individual sieves provide.
Furthermore, the original sieves can be recovered by taking either  $\Omega=\Z^n$ or $z=1$.
\begin{remark}[Necessity of assumptions] The case $\Omega=Y$ shows that extra 
assumptions are necessary.
This is why we added  the   assumption  $\limsup \omega(p) \neq 1 $, which prevents the case $\Omega=Y$.
Indeed,    if $\Omega (\F_p)\subseteq Y(\F_p)$ for infinitely many primes then  $1-\omega(p) =O(p^{-k})$ by 
  the Lang--Weil estimates~\cite{LW54},  hence, the assumption $\lim \sup \omega(p)\neq  1$.
In cases where $\lim \sup \omega(p)= 1$, the large sieve alone typically provides a satisfying upper bound. \end{remark}
\begin{remark}[A no-assumptions version] We later give a version with no assumptions in 
Theorem~\ref{thm:sum}.
Friendly versions of Theorem~\ref{thm:sum} are given in Corollaries~\ref{cor:asympt}-\ref{cor:pk}.
  Theorem~\ref{thm:generalis} is a special case of any of these results, however,
it is easy to use and it suffices in  most     cases. \end{remark}
\subsection{$W$-trick for local soluble equations}\label{s:probloc}
The problem of estimating the probability that a variety is everywhere locally soluble has recently attracted much attention,
see the table in the introduction of~\cite{conconcon} for some of the results. As there is no uniform treatment for all cases of this question 
it is desirable to have      a general framework for this type of question. We propose here a variant of the Green--Tao
$W$-trick  which consists of the following three steps:
\begin{enumerate}\item \textbf{Simplify}: 
use the   geometric-large sieve from Theorem~\ref{thm:generalis} to  ensure that $100\%$ of everywhere locally soluble varieties 
have ``simple'' coefficients, i.e. square-free, coprime, e.t.c., with respect to all primes $p>z$ for some $z\to\infty$,
\item  \textbf{Divide}: partition the coefficients of the everywhere locally soluble varieties in arithmetic progressions modulo a 
multiple $W$ of all primes $p\leq z $,
\item  \textbf{Rule}: use the fact that the coefficients  of the remaining varieties are arithmetically simple
to prove asymptotics for the number of everywhere 
locally soluble varieties in each arithmetic progression modulo $W$.  
\end{enumerate} The third step is   a Siegel--Walfisz type of question for the equidistribution of everywhere locally soluble varieties in arithmetic progressions. To control error terms it is desirable to work only with small moduli $W$, which can be ensured by
  taking $z=z(B)$  tend to 
infinity very slowly; this does not cause problems owing to    Theorem~\ref{thm:generalis}. We will illustrate these steps in our proof of 
Theorem~\ref{thm:analytic}. \begin{remark}A  convenient side  of  this  approach is that the leading constant automatically
 comes   factored as an Euler product 
since the asymptotic contribution of each progression in the last step seems to be independent of the progression.
\end{remark}
\begin{acknowledgements}
Sections of this paper were written while the authors were visiting the Max Planck Insitute for Mathematics in Bonn and Carla and Marco’s place in Anacapri. We are deeply grateful to them for the warm hospitality, to the former also for financial backing and to the latter for the environment replete with fresh caciotta.
We thank Tim Browning and Dan Loughran for useful comments on an earlier version of this paper.
\end{acknowledgements}

\section{The geometric-large sieve}

We prove 
Theorem~\ref{thm:generalis} by
replacing the treatment of the small primes in Bhargava's proof of~\cite[Theorem 3.3]{1402.0031}
by arguments   from the  large sieve. We start with the following lemma:
\begin{lemma}\label{corollarge}
Keep the setting of Lemma~\ref{lemser}.
Assume that
  for every prime $p\leq B^{1/4m}$
we are given a set $S_{p} \subseteq (\Z/p^m \Z)^n $.
Then the number of $\b x \in \Omega\cap \Z^n$ for which there exists a prime 
$p\in (z, B^{1/4m} ]$  
with $ \b x \md{p^m} \in S_{p}$ is at most 
\[ 
\l(
\sum_{p\in (z, B^{1/4m} ]} \frac{  \#\{S_{p} \cap \Omega(\Z/p^m \Z)\} }  {  \#\Omega(\Z/p^m \Z) }
\r)
\frac{(2B)^n}{ L(B^{1/4m}) }
.\]

\end{lemma}
\begin{proof} By the union bound 
we get 
\[\leq \sum_{p\in (z, B^{1/4m} ]}
\#\{\b x \in \Omega\cap \Z^n :\b x \md{p^m} \in S_{p } \}
.\]
Now we use Lemma~\ref{lemser} with 
$\omega_p(\ell):= \omega(\ell)$ for all $\ell\neq p $ and with 
\[p^{nm}(1-\omega_p(p) ):= \#\{S_{p} \cap \Omega(\Z/p^m \Z)\} .\]
We obtain
\[
  \sum_{p\in (z, B^{1/4m} ]}
\#\{\b x \in \Omega\cap \Z^n :\b x \md{p^m} \in S_{p } \}
\leq (2B)^n  
\sum_{p\in (z, B^{1/4m} ]} \frac{1}{M_p(B^{1/2m})}
\] where 
\[
M_p(t)=\sum_{q\leq t } \mu^2(q) \prod_{\ell\mid q } \frac{\omega_p(\ell) }{1-\omega_p(\ell) } .\]
Let 
\[g(q)=\mu^2(q) \prod_{\ell\mid q } \frac{\omega (\ell) }{1-\omega (\ell) } 
,
L_p(t)=\sum_{\substack{ q\leq t \\ p\nmid q  }} g(q) 
\]
so that for $p\leq \sqrt t $ one has 
\[M_p(t)=
L_p(t)+\frac{p^{nm}-  \#\{S_{p} \cap \Omega(\Z/p^m \Z)\} }{  \#\{S_{p} \cap \Omega(\Z/p^m \Z)\} } L_p(t/p)
\geq 
\frac{p^{nm} }{  \#\{S_{p} \cap \Omega(\Z/p^m \Z)\} } L_p(\sqrt t ),
\]
where we used $\min\{L_p(t), L_p(t/p)\}\geq L_p(\sqrt t ) $ that is implied by $t/p \geq \sqrt t $.
 Note that we also have 
\[
L(\sqrt t) =L_p(\sqrt t) + \frac{\omega(p) }{1-\omega(p) } L_p(\sqrt t/p) 
\leq L_p(\sqrt t) + \frac{\omega(p) }{1-\omega(p) } L_p(\sqrt t)=\frac{L_p(\sqrt t)}{1-\omega(p)} 
 \]hence \[M_p(t)\geq \frac{p^{nm} }{  \#\{S_{p} \cap\Omega(\Z/p^m \Z)\} } (1-\omega(p))
L(\sqrt t )=L(\sqrt t )
\frac{  \# \Omega(\Z/p^m \Z) }{  \#\{S_{p} \cap \Omega(\Z/p^m \Z)\} }  
.\]This is sufficient.
\end{proof}The following result is our most general combination of the geometric and the large sieves. It makes no assumptions on $\omega(p)$. 
\begin{theorem}\label{thm:sum}Keep the setting of 
 Lemmas~\ref{lembharg}-\ref{lemser} and define 
$$ \c E= \sum_{p\in (z, B^{1/4m} ]} \frac{  \#\{Y(\Z/p^m \Z) \cap \Omega(\Z/p^m \Z) \} }  {  \#\Omega(\Z/p^m \Z) }
.$$ Then for all $B,z\geq 2 $
we have  \[\#\{\b x \in B U\cap \Omega: \b x \md p \in Y(\F_p) \textrm{ for some prime } p>z \}=O\l(
\c E\frac{B^n}{L(B^{1/4m})}+ B^{n-\frac{(k-1)}{4m} }  \r), \] where  the implied constant depends only on $U$ and $Y$.
\end{theorem} \begin{proof}  We can clearly assume that $k\geq 2 $ since otherwise the second error term dominates.
By Lemma~\ref{lemser} we infer \[ \#\{\b x \in B U\cap \Omega: \b x \md p \in Y(\F_p) \textrm{ for some prime } p>B^{1/4m}  \}
= O\l(\frac{B^{n-\frac{(k-1)}{4m} }}{ \log B }+B^{n-k+1}\r), \]which is $\ll B^{n-\frac{(k-1)}{4m}}$. This is sufficient if $z\geq B^{1/4m} $. When 
$z< B^{1/4m} $   it suffices to prove that  \[\#\{\b x \in B U\cap \Omega: \b x \md p \in Y(\F_p) \textrm{ for some } p\in (z, B^{1/4m} ]  \} 
 =O\l( \c E \frac{B^n}{L(B^{1/4m})} + B^{n-\frac{(k-1)}{4m} }   \r). \] This follows directly by Lemma~\ref{corollarge}
with $S_p=Y(\Z/p^m\Z)$.  \end{proof}

If the sets $Y(\Z/p^m \Z) ,  \Omega(\Z/p^m \Z)$
are `independent'  for sufficiently many primes $p$,
our strategy always gives 
a big saving over the large sieve.
We make this precise in the next result:
\begin{corollary}
\label{cor:asympt}
Keep the setting of 
 Lemmas~\ref{lembharg}-\ref{lemser} and 
assume that 
$$\lim_{t\to\infty} \frac{L(t)}{t^{k-1}}=0 
\ \ \textrm{  and  } \ \ 
   \sum_{p\textrm{ prime } } \frac{  \#Y(\Z/p^m \Z) \cap \Omega(\Z/p^m \Z)  }  {  \# \Omega(\Z/p^m \Z)  }
<\infty .$$ Then for any function $\xi:[1,\infty)\to \R$ with $\lim_{t\to\infty}\xi(t) =+\infty $ we have 
\[
\lim_{B\to\infty }
\frac{\#\{\b x \in B U\cap \Omega: \b x \md p \in Y(\F_p) \textrm{ for some prime } p>\xi(B) \}}{
B^n/L(B^{1/4m})}
=0
 .\]
\end{corollary}
\begin{proof}By Theorem~\ref{thm:sum} the quotient in the corollary is 
\[\ll \frac{L(B^{1/4m})}{B^{(k-1)/4m} }+\sum_{p>\xi(B) } \frac{  \#Y(\Z/p^m \Z) \cap \Omega(\Z/p^m \Z)  }  {  \# \Omega(\Z/p^m \Z) }
.\] Our assumptions ensure that both terms vanish as $B\to\infty$.
\end{proof}
Next, we give
a
version of Theorem~\ref{thm:sum}
that is easier to use. It briefly states that if 
$\Omega(\Z/p^m \Z)$ is reasonably large, one always gets a substantial saving over the large sieve.
 \begin{corollary}
\label{cor:pk}
Keep the setting of 
 Lemmas~\ref{lembharg}-\ref{lemser} and 
assume that 
$$\liminf_{p\to\infty} \frac{ \# \Omega(\Z/p^m \Z)  }{p^{nm-k+1}\log p}\neq 0  .$$ 
Then for any $B,z\geq 2 $ we have 
\[
\#\{\b x \in B U\cap \Omega: \b x \md p \in Y(\F_p) \textrm{ for some prime } p>z \}
\ll
\frac{1}{(\log z )}
\frac{B^n}{L(B^{1/4m})}
+
B^{n-\frac{(k-1)}{4m}}
,\] where the implied constant depends only on $U,Y$ and  the value of $\liminf$.
\end{corollary}
\begin{proof}
By     the Lang--Weil estimates~\cite{LW54}
 we have 
\beq{eq:haswl}{\#Y(\Z/p^m \Z) \cap \Omega(\Z/p^m \Z)  \leq \#Y(\Z/p^m \Z)  =p^{n(m-1)}\#Y(\F_p) 
\ll p^{mn-k}
.}
Combining this with our assumption, shows that there exists a positive constant $c$ such that 
$$\# \Omega(\Z/p^m \Z)  
\geq  c  p \log p
\#Y(\Z/p^m \Z) \cap \Omega(\Z/p^m \Z)  
 ,
$$ hence, $\c E\leq c 
  \sum_{p>z} \frac{1}{p\log p}=O_c(1/ \log z )
$,  by the Prime Number Theorem. 
\end{proof}

\subsection{Proof of Theorem~\ref{thm:generalis}}
By~\eqref{eq:haswl}
we see that the quantity $\c E $ in Theorem~\ref{thm:sum} is 
\[
\c E\ll \sum_{p>z} \frac{1}{p^{-nm+k} \#\Omega(\Z/p^m\Z)}
.\]
Using the assumption $\gamma=\limsup_{p\to\infty}\omega(p)$ is not $  1 $ 
we note that 
\[\#\Omega(\Z/p^m\Z) = p^{nm} (1-\omega(p) ) \gg_\gamma p^{nm}
\] for all sufficiently large primes $p$.   
 Therefore, $\c E\ll_{\gamma} \sum_{p>z}p^{-k}\ll z^{-k+1}(\log z)^{-1} $.
\qed

\subsection{Proof of Theorem~\ref{thm:loughrnsmt}}
We use the setting of~\cite[\S 4.2.3]{MR3568035}  where $\omega(p)$ is defined and is shown that $L(T) \gg (\log T)^{\Delta(f)}$.
The proof follows directly from Theorem~\ref{thm:generalis} since it is proved in~\cite[Lemma 3.3]{MR4269677} that $\omega(p)\ll 1/p$.

\section{Character sums}   \label{s:checkingsolubility} 
In this section we prove Theorem~\ref{thm:analytic}. Recall the three steps      in \S\ref{s:probloc}. The first step is in \S\ref{lemslargeprim}:
using  the geometric-large sieve we show that only pairs of integers 
that   jointly divisible by small powers of primes contribute. 
In the second step in \S\ref{s:periodicity} we use this information to partition into almost-primitive 
progressions modulo  some  integer  $W$ which  is divisible by all   primes below an arbitrary $z$ 
that grows slowly to infinity with $B$. In the last step in \S\ref{s:methodiwafrid}  we use the method of Friedlander--Iwaniec~\cite{MR2675875}
to prove an asymptotic inside each progression. The main term treatment and the final steps in the proof of the asymptotic are   in \S\ref{lem:joujou2}. 

Throughout this section we choose and fix   $a,b\in \{1,-1\}$, $\c P$ will be a  subset of the primes and   $$  S_{\c P}:=\{(s,t)\in \Z^2: as> 0,bt > 0, 
(s,t)_{\Q_p}=1 \ \text{for every prime }  p\in \c P\}.$$ \subsection{First step: simplify} 
\begin{lemma}\label{lemslargeprim}For      $B, z\geq 1$ the number of $(s,t)\in S_\c P\cap [-B,B]^2$ for which there exists a prime $\ell>z$ such that $\ell \mid (s,t)$ is$$ \ll  \frac{1}{z (\log z) } \frac{B^2}{(\log B)^{\varpi} } + B^{2-\frac{1}{8} }  ,$$ where the implied constant is independent of $B$ and $z$. 
\end{lemma}\begin{proof}We use Theorem~\ref{thm:generalis} with $\Omega=S_\c P\cap [-B,B]^2$, $m=2, n=2,U=[-1,1]^2$ and $Y=\{(0,0)\}$.
To bound $\#\Omega(\Z/\ell \Z)$ for $\ell\geq 5 $   in $\c P$  we consider separately the contribution of the cases 
\begin{itemize}\item $\ell^2 \mid st $, \item $\ell \mid s, \ell^2\nmid s, \ell \nmid t$,
\item $\ell \mid t, \ell^2\nmid t, \ell \nmid s$,\item $\ell\nmid st$.
\end{itemize}  In the second case, $t$ must reduce to  a square in $\F_\ell$, hence, we obtain 
 $$\ell^{4}(1-\omega(\ell) ) = \#\Omega(\Z/\ell^2 \Z)\leq 3\ell^2+2(\ell-1) \frac{\ell(\ell -1)}{2}+ \ell^2 (\ell-1)^2\Rightarrow 
  \frac{ 1}{\ell} \l(1-\frac{4}{\ell} \r)  \leq    \omega(\ell)   .$$ To     upper-bound   $\omega(\ell)$ we note that the last case gives 
 $\ell^{4}(1-\omega(\ell) ) = \#\Omega(\Z/\ell^2 \Z)\geq \ell^2 (\ell-1)^2 $, hence, 
$ \omega(\ell) \leq 2/\ell$. For primes $\ell \notin \c P$ or $\ell=2,3$ we use the trivial bound $\omega(\ell)\geq 0$. Thus, $\limsup \omega(\ell) =0$, hence,  Theorem~\ref{thm:generalis} provides the following bound for the the quantity in our lemma:
$$ \ll  \frac{1}{z (\log z) } \frac{B^2}{L(B^{\frac{1}{8}})} + B^{2-\frac{1}{8} } 
,$$ where  $$L(T)\geq \sum_{\substack{ q\leq T \\\ell\mid q\Rightarrow \ell\ in  \ \c P, \ell\geq 5 }} \frac{\mu^2(q)}{q}
\prod_{\ell \mid q }   \l(1-\frac{4}{\ell} \r)  \frac{1}{1-\frac{1}{\ell} (1-\frac{4}{\ell} ) }
\gg \prod_{\substack{ 5\leq \ell \leq T\\ \ell \in \c P }}\l(1-\frac{1}{\ell} \r)\gg (\log T)^{-\varpi},$$
where we used~\cite[Theorem 14.3]{koukou} and~\eqref{AssumptionsonAAA} for the character $\chi=1$. \end{proof}
\begin{lemma}\label{bachfrenchsuites}For any $m\in \N$, prime $p$ and $B\geq p^{8m}$,
the number of $(s,t)\in S_\c P\cap [-B,B]^2$ such that  $p^m \mid s$ is $$ \ll  \frac{m^{\varpi}}{p^m } \frac{B^2}{(\log B)^{\varpi} }  ,$$ 
where the implied constant is independent of $B,p,m$ and $z$. 
\end{lemma}\begin{proof} We use Lemma~\ref{lemser} with $n=2$. For primes $\ell \in \c P\setminus \{2,3,p\}$
one can use arguments  as in the proof of the proof of Lemma~\ref{lemslargeprim} to see that  $ 1-4 \ell^{-1}  \leq \omega(\ell) \ell  $.
For $\ell=p$  we trivially have    $p^{2m}(1-\omega(p) ) = \#\Omega(\Z/p^m \Z)\leq p^m$, hence, $\omega(p) \geq 1-p^{-m}$.
Therefore, if  $p\leq T$ we obtain $$\sum_{q\leq T } \mu^2(q)\prod_{\ell \mid q } \frac{\omega(\ell )}{ 1-\omega(\ell ) }
\geq \frac{\omega(p )}{ 1-\omega(p ) } \sum_{\substack{ t\leq T/p \\ \ell \mid t\Rightarrow \ell \in \c P\setminus \{2,3,p\} }} 
\mu^2(t)\prod_{\ell \mid t } \frac{\omega(\ell )}{ 1-\omega(\ell ) }
.$$Since $1/(1-\omega(\ell))\geq 1 $,  $\omega(p) \geq 1/2$ and $1-\omega(p)\leq p^{-m}$ we get the lower bound 
$$\frac{p^m}{2} \sum_{\substack{ t\leq T/p \\ \ell \mid t\Rightarrow \ell \in \c P\setminus \{2,3,p\} }} 
\frac{\mu^2(t)}{t} \prod_{\ell \mid t }  \l(1-\frac{4}{\ell} \r) \gg p^m \prod_{\substack{ \ell \leq T/p \\ \ell \in \c P\setminus \{2,3,p\} }} 
\l(1+\frac{1}{\ell}\r) \gg p^m \l(\log \frac{T}{p} \r)^{\varpi},$$by~\cite[Theorem 14.3]{koukou} and~\eqref{AssumptionsonAAA}. If $p\leq T^{1/2}$
the lower bound is $\gg p^m (\log T)^{\varpi}$ and we may use this with $T=B^{1/2m}$ together with Lemma~\ref{lemser} to conclude the proof.\end{proof}\begin{lemma}\label{bachpartitas}For      $B, z\geq 2$ the number of $(s,t)\in S_\c P\cap [-B,B]^2$ for which there exists a prime $p>z$ such that $p^2\mid s$ or $p^2\mid t $ is $$ \ll  \frac{1}{z (\log z) } \frac{B^2}{(\log B)^{\varpi}} + B^{2-1/16} ,$$ where the implied constant is independent of $B$ and $z$. 
\end{lemma}\begin{proof}The contribution of $p>B^{1/16}$ is bounded trivially by $$ \ll \sum_{p>B^{1/16}}\frac{B^2}{p^2}\ll B^{2-1/16}.$$For the remaining range we can use Lemma~\ref{bachfrenchsuites} with $m=2$ and sum over $p$ in $ (z,B^{1/16}]$.\end{proof}
Define for each prime $\ell \leq z$ the integer \beq{apologynotrequested}{k_\ell= \l[\frac{\log z }{\log \ell} \r].}
\begin{lemma}\label{bachpartitas23}For      $2\leq z \leq B^{1/16}$ the number of $(s,t)\in S_\c P\cap [-B,B]^2$ for which there exists a prime $\ell\leq z $ 
such that $ \ell^{1+k_\ell}\mid s$ or $\ell^{1+k_\ell}\mid t $ is $$ \ll  \frac{1}{z ^{1/2} (\log z) } \frac{B^2}{(\log B)^{\varpi}}  ,$$ where the 
implied constant is independent of $B$ and $z$.  \end{lemma}\begin{proof}We have $\ell^{1+k_\ell} \leq \ell z\leq z^2 \leq B^{1/8}$, hence, by 
Lemma~\ref{bachfrenchsuites} with $m=1+k_\ell$ we obtain the bound
$$\ll \frac{B^2}{(\log B)^{\varpi} }\sum_{\ell \leq z} \frac{(1+k_\ell)^{\varpi}}{\ell^{1+k_\ell}}= \frac{B^2}{(\log B)^{\varpi}} \sum_{1\leq k \leq \frac{\log z}{\log 2} } 
(1+k)^{\varpi} \sum_{  z^{1/(k+1)} < \ell \leq z^{1/k}   } \frac{1}{\ell^{1+k}}
.$$  By partial summation and the prime number theorem  we  bound the contribution of $k=1$ by 
$$ \ll  \sum_{\ell>\sqrt z} \frac{1}{\ell^2} \ll \frac{1}{\sqrt z(\log z)}.$$ 
Noting that $\varpi\leq 1 $ and using the estimate $$ \sum_{m\in \N, m>y}m^{-k-1} \ll \int_y^\infty t^{-k-1}\mathrm d t +y^{-k-1} \ll \frac{1}{k y^k }+y^{-k-1}$$
that holds with   absolute implied constants, shows that the sum over $k\neq 1$ is  \[ \ll \sum_{2\leq k\leq \frac{\log z }{\log 2 } }  (k+1) 
 \l(\frac{1}{k z^{k/(k+1)}   }+ \frac{1}{z}   \r) \ll  \sum_{2\leq k\leq \frac{\log z }{\log 2 } } \frac{1}{ z^{ k/(k+1)} }    
+ \sum_{2\leq k\leq \frac{\log z }{\log 2 } } \frac{k}{z}  .\] This is
\[\ll \frac{1}{z^{3/4} } \sum_{2\leq k\leq \frac{\log z }{\log 2 } } 1 + \sum_{2\leq k\leq \frac{\log z }{\log 2 } } \frac{k}{z} \ll 
\frac{\log z}{z^{3/4}} + \frac{(\log z)^2}{z}, \]  which is satisfactory.\end{proof}
Combining Lemmas~\ref{lemslargeprim}-\ref{bachpartitas}-\ref{bachpartitas23} yields the following:
 \begin{lemma}\label{bachpartitasbsd7}With $k_\ell$ as in~\eqref{apologynotrequested}  and any    $2\leq z \leq B^{1/16}$ we have 
\[N(B;\c P) =\#\left\{\left(s,t\right) \in S_\c P:\begin{array}{l} 0<as,b t\leq B  \\\ell \leq z  \Rightarrow \ell^{1+k_\ell}\nmid s, \ell^{1+k_\ell}\nmid t  
,\\ \ell >z \Rightarrow  \ell \nmid \gcd(s,t)    , \ell ^2\nmid s, \ell ^2\nmid t     \end{array} \right\} +O\l(\frac{1}{z^{1/2} (\log z) } \frac{B^2}{(\log B)^{\varpi} }\r)
,\]where the  implied constant is independent of $B$ and $z$.  \end{lemma}
\subsection{Second step: divide}\label{s:periodicity}We shall need the following periodicity property  of the Hilbert symbol:
Let $p$ be an  odd prime, $k\geq 2 $ be an integer and let  $s,t \in \Z$ satisfy  $v_p(s),v_p(t)\leq k$.
If $\sigma, \tau $ are integers in the range 
$1\leq \sigma,\tau \leq p^{1+k_p} $ 
that satisfy $(s,t)\equiv (\sigma,\tau)\md {p^{1+k} }$, 
then   $v_p(\sigma)=v_p(s)$,
$v_p(\tau )=v_p(t)$ 
and   
$ s p^{-v_p(s) }\equiv \sigma p^{-v_p(\sigma)}\md {p} $, 
$ t p^{-v_p(t) }\equiv \tau p^{-v_p(\tau )}\md {p} $.
In particular, $$ (s,t)_{\Q_p} =\l(\frac{-1}{p} \r)^{v_p(s)v_p(t)}
\l(\frac{s p^{-v_p(s)}}{p} \r)^{v_p(t)}
\l(\frac{t p^{-v_p(t) }}{p} \r)^{v_p(s)}
=(\sigma,\tau)_{\Q_p}.$$
For $p=2 $ and integers $s,t $ with $v_2(s),v_2(t)\leq k$
we let 
$\sigma, \tau $ be  integers in the range 
$1\leq \sigma,\tau \leq 2^{3+k } $ with  $(s,t)\equiv (\sigma,\tau)\md {2^{3+k} }$.
It is then easy to see that $v_2(s)=v_2(\sigma),v_2(\tau )=v_2(t)$ 
and that  $s 2^{-v_2(s)} \equiv \sigma 2^{-v_2(\sigma)} \md{8}, 
\tau  2^{-v_2(\tau )} \equiv t  2^{-v_2(t )} \md{8}$.
In particular,  $(s,t)_{\Q_2}=(\sigma,\tau)_{\Q_2} $. 

With $k_\ell$ as in~\eqref{apologynotrequested}  we   define  $$ W=2^{3+k_2} \prod_{\ell \textrm{ prime in } (2,z]  } \ell^{1+k_\ell} .$$   
\begin{lemma}\label{saladthencoffee} For any    $2\leq z \leq B^{1/16}$ we have  
\[N(B;\c P)
=
\sum_{\substack{ (\sigma ,\tau ) \in (\Z/W\Z)^2   \\ v_\ell(\sigma),v_\ell(\tau)\leq k_\ell \forall \ell\leq z   }}
\c M_{\sigma ,\tau }( B;z)+O\l(\frac{1}{z^{1/2} (\log z) } \frac{B^2}{(\log B)^{\varpi} }\r)
, \]  where  the sum is subject to the extra condition 
$(\sigma ,\tau  )_{\Q_\ell}=1$ for all primes $ \ell \in \c P \cap[1,z] $ and 
\[\c M_{\sigma ,\tau }( B;z):=\# \left\{\left(s,t\right) \in \Z^2: \begin{array}{l}
(s,t)\equiv (\sigma, \tau ) \md W,\\ 0<as, bt\leq B   , \\
\ell >z \Rightarrow  \ell^2 \nmid st, \\ \ell>z, \ell \in \c P\Rightarrow (s,t)_{\Q_\ell}= 1  \end{array}
\right\}
.\]\end{lemma}
\subsection{Third step: rule}\label{s:methodiwafrid}
We use the Hasse principle to bring in  
explicit expressions. \begin{lemma}\label{lem:fridiwanchilvert}We have $$\c M_{\sigma,\tau}(B;z)=
  \sum_{\substack{ \b d , \b e \in \N^2,  \ell\mid d_1e_1\Rightarrow \ell\in \c P
  \\  d_1 d_2 \leq B/\sigma',  e_1e_2 \leq B /\tau'  }}  
\frac{\mu(d_1d_2e_1e_2)^2}{2^{\#\{ \ell\in \c P: \ell \mid d_1d_2e_1e_2 \}} }
\l(\frac{b \tau'  e_2 }{d_1}\r) \l(\frac{a \sigma'   d_2 }{e_1}\r)
(-1)^{\frac{(d_1-1)(e_1-1)}{4}},$$ where the sum   over $\b d ,\b e $ is subject to  $  d_1 d_2 \equiv a\sigma/\sigma'    \md { W /\sigma'  } ,
 e_1 e_2\equiv b\tau/\tau'    \md { W / \tau'   } $  and the constants $\sigma',\tau'$ are defined by 
$$\sigma':=\prod_{p\leq z}p^{v_p(\sigma)},  \tau':=\prod_{p\leq z}p^{v_p(\tau)}.$$
\end{lemma}\begin{proof} Let us start by factoring    $s,t$ as a product of primes exceeding $z$ and primes below $z$, namely,
\beq{eq:hol2}{s=a \sigma' s_1, t=b \tau' t_1, } where  
$$ s_1:=\prod_{p>z}p^{v_p(s)},  t_1:=\prod_{p>z}p^{v_p(t)},$$ due to  $v_p(s)=v_p(\sigma)$ and $v_p(t)=v_p(\tau)$. Since 
$v_p(\sigma),v_p(\tau)<v_p(W)$  for all  $p\mid W $ we get  \beq{eq:progres}{  s_1\equiv\frac{a\sigma}{\sigma' }  \md {\frac{W  }{\sigma' } } ,
 t_1\equiv\frac{b\tau}{\tau' }  \md {\frac{W  }{\tau' } }  .} Any integer congruent to $b\tau/\tau' \md {W/\tau'}$ must be coprime to $W$. This is because for each $p\mid 
W$ we have  $v_p(\tau)=v_p(\tau' )$ and $v_p(W/\tau')\geq 1 $.    This gives 
\beq{eq:hol}{\c M_{\sigma,\tau}(B;z)=  \# \left\{\left(s_1,t_1\right) \in \N^2:
\begin{array}{l}   s_1   \leq B/\sigma' ,      t_1 \leq B /\tau',   \\
\eqref{eq:progres}, \mu^2(s_1 t_1 )=1,\\z<\ell\in \c P\Rightarrow (a\sigma's_1,b\tau't_1)_{\Q_\ell}= 1 \end{array}
\right\} .} The condition $z<\ell\in \c P\Rightarrow (a\sigma's_1,b\tau't_1)_{\Q_\ell}= 1$ has indicator function 
 \[  \prod_{\substack{\ell \mid s_1 \\ \ell\in \c P   }}\frac{1+\l(\frac{b\tau't_1}{\ell}\r)}{2}
\prod_{\substack{\ell \mid t_1 \\ \ell \in \c P }}\frac{1+\l(\frac{a\sigma's_1}{\ell}\r)}{2}
= 2^{-\#\{ \ell\in \c P: \ell \mid s_1   t_1 \}} \sum_{\substack{ \b d,\b e  \in  \N^2 \\d_1d_2=s_1,
e_1e_2=t_1  \\ \ell\mid d_1e_1\Rightarrow \ell\in \c P}} \l(\frac{b \tau' t_1}{d_1}\r) \l(\frac{a \sigma's_1 }{e_1}\r) ,\] 
where the products over primes $\ell \mid s_1$ and $\ell\mid t_1$ do not contain the condition $\ell>z$ since $s_1t_1$ is coprime to $W$.   
We can make use of  quadratic reciprocity to simplify $(\frac{d_1}{e_1})(\frac{e_1}{d_1})$, which is allowed 
owing to the fact that  for $z>2$ the positive integers $d_1,e_1$ are coprime to $W$, hence odd.
Substituting into~\eqref{eq:hol} concludes the proof of the lemma.\end{proof}
 Next, we reduce the range of the sum over $\b d,\b e$ by using the large sieve for quadratic characters.
\begin{lemma}\label{housephoneisringingandwakingmeup}The contribution to the sum in Lemma~\ref{lem:fridiwanchilvert} 
of the terms  for which  we have 
$  \max\{d_1,e_1\}>(\log B)^{20}  $ or $ \max\{d_2,e_2\}> (\log B)^{20}$   
  is $ \ll B^2(\log B)^{-7/6}$, where the implied constant  is absolute.
\end{lemma} \begin{proof} Let $C=20$. The contribution of the terms with $\max\{d_2,e_1\} \leq (\log B)^C$ is 
  \[ \ll  \sum_{   d_2, e_1 \leq (\log B)^C} \bigg| \sum_{\substack{   d_1,   e_2 \ \in \  \N, \ell\mid d_1 \Rightarrow \ell\in \c P
\\  d_1\leq B/(\sigma'd_2), e_2\leq B/(\tau'e_1) \\   \gcd(d_1 e_2  ,d_2e_1  )= 1    } }
   \l(\frac{  e_2}{d_1}\r)  \frac{\mu(d_1 )^2\mu( e_2)^2}{ 2^{\#\{\ell\in \c P: \ell \mid d_1  e_2\}}}
  \l(\frac{b\tau'   }{d_1}\r)  (-1)^{\frac{(d_1-1)(e_1-1)}{4}} \bigg| ,\] where the sum over $d_1,e_2$ is subject to 
further congruence conditions as in Lemma~\ref{lem:fridiwanchilvert}.
These congruence conditions imply that $d_1e_2$ is odd, hence,~\cite[Lemma 2]{MR2675875} can be employed.
It yields   \[ \ll  B^{11/6} (\log B)^{7/6} \sum_{   d_2, e_1 \leq (\log B)^C} \l(   \frac{ 1}{  d_2e_1^{5/6}  } +    \frac{1}{e_1d_2^{5/6} } \r) 
  \ll B^{13/7} .\] A symmetric argument supplies the same bound for the      contribution of 
$\max\{d_1,e_2\}\leq (\log B )^C$. 
When  $\min\{d_1, e_2\}> (\log B )^C$, one will have  $d_2  \leq B /(\log B )^C$ and  $e_1\leq B / (\log B )^C$.
As before, the contribution is then seen to be  $$\ll  B^{11/6} (\log B )^{7/6} \sum_{  d_2 ,  e_1\leq B / (\log B)^C  }
\l(   \frac{1}{ d_2  e_1^{5/6}  } +    \frac{ 1}{e_1 d_2^{5/6} } \r) \ll  \frac{B^2 }{(\log 3B)^{-13/6+C/6}}.$$ The cases with 
$\min\{ d_2,e_1\}> (\log B )^C$ are treated in a similar manner.

Assume that $d_1 \leq (\log B )^C$. Then the terms with $e_2 \leq   (\log B  )^C$  have been treated, thus, we are left with the range
$e_2 >   (\log B )^C$. If $e_1 >   (\log B  )^C$ then we  must have $d_2 \leq (\log B )^C$, which shows that  
 $\max\{ d_1, d_2\} \leq (\log B )^C$. Then $d_1d_2 \leq (\log B )^{2C}$, hence, this   contributes  
$$ \ll \sum_{\substack{ e_1e_2 \leq B \\ d_1d_2 \leq (\log B )^{2C}   }} 1\ll B (\log B )^{1+2C} ,$$ which is acceptable.
In conclusion, if $d_1 \leq (\log B  )^C$ then one must have $e_1 \leq    (\log B  )^C$. 
A similar argument shows that  if $d_2 \leq (\log B  )^C$ then only the cases with  $e_2 \leq    (\log B  )^C$ do not contribute into the error term.
\end{proof}  
The contribution of the  terms with  $\max\{d_1,e_1\}\leq (\log B)^{20}$ and $(d_1,e_1) \neq (1,1)$  towards the sum in 
Lemma~\ref{lem:fridiwanchilvert} is $$ \ll \sum_{\substack{d_1,e_1 \leq (\log B)^{20} \\ \gcd(d_1e_1,W)=1, (d_1,e_1) \neq (1,1)}} 
\bigg|  \sum_{\substack{ d_2,e_2 \in \N, \gcd(d_2e_2,d_1e_1)=1   \\  d_1 d_2 \leq B/\sigma',  e_1e_2 \leq B /\tau'  }}  
\frac{\mu(d_2)^2 \mu(e_2)^2}{2^{\#\{ \ell\in \c P: \ell \mid  d_2 e_2 \}} } \l(\frac{ e_2 }{d_1}\r) \l(\frac{  d_2 }{e_1}\r)
\bigg|,$$ where the double sum over $d_2,e_2$ is subject to the congruence conditions in the analogous expression in Lemma~\ref{lem:fridiwanchilvert}.
Note that these congruences are primitive since $\gcd(d_1e_1,W)=1$. A straightforward modification of the case of the non-principal character case 
of~\cite[Corollary 2]{MR2675875} shows that if $d_1\neq 1$ then the sum over  $d_2$ is $\ll \tau(d_1 e_1e_2) e_1 B (\log B)^{-C}$ for any fixed $C>0$.
The overall contribution is    $$ \ll \frac{B}{(\log B)^C} \sum_{ d_1,e_1 \leq (\log B)^{20}} \tau(d_1) \tau(e_1) e_1   \sum_{  e_2   \leq B    } 
\tau(e_2) \ll B^2 (\log B)^{-C+44}, $$ which is $\ll B^2(\log B)^{-2}$   when     $C=46$.
A symmetric argument gives the same bound when $e_1 \neq 1 $.

  \begin{lemma}\label{lem:panolino} Let $\beta$ be a positive odd square-free number that is 
large enough so that assumption $(3)$ of Theorem~\ref{thm:analytic} holds. 
For  any $N\geq \exp(\beta^{1/160}) $, any $q$ coprime to $\beta$ with $\exp(\{\log \c L(q)\}^{5/4} )\leq N$, 
any   $a\in (\Z/q\Z)^*$ 
and any $m\in \N$ with   $\omega(m) \leq \exp ( \sqrt{\log N } )$
one has  $$\sum_{\substack{n\leq N, \gcd(n,m)=1  \\ p\mid n\Rightarrow p\in \c P\\ n\equiv a \md q } } \frac{\mu(n)^2}{ \tau(n) } \l(\frac{n}{\beta}\r) \ll 
\frac{N}{(\log N)^{2024} },$$ where the implied constant depends at most on $\c P$. \end{lemma}\begin{proof} By orthogonality of characters we can write 
the sum as $$\frac{1}{\phi(q)}\sum_{\psi\md q } \overline{\psi(a)} 
\sum_{\substack{n\leq N, \gcd(n,m)=1  \\ p\mid n\Rightarrow p\in \c P} } \frac{\mu(n)^2}{ \tau(n) } \psi(n) \l(\frac{n}{\beta}\r).$$
Define the multiplicative function  $$ f(n)=\frac{\mu(n)^2}{ \tau(n) } \l(\frac{n}{\beta}\r) \mathds 1 (p\mid n\Rightarrow p\in \c P)
\mathds 1 (p\mid n\Rightarrow p\nmid m).$$  
We have $$ \sum_{p\leq T} f(p) \log p=\frac{1}{2} \sum_{\substack{ p\leq T\\p\in \c P}}\psi(p) \l(\frac{p}{\beta}\r)\log p+O\l(\omega(m) \log T\r)$$
with an absolute implied constant. Assume that $T\geq \exp\{(\log N)^{3/4}\}$ so that the assumed bound on $\omega(m)$ 
yields $\omega(m) \log T \ll T^{3/4}$. Define $$ Q:=  \max \l\{\exp\{(\log N)^{3/4}\}, \c L(q_\psi), \exp(\beta^{1/200})\r\}.$$
Since $\beta$ is large enough, the third assumption in Theorem~\ref{thm:analytic}  shows for each fixed $A>0$  one 
  has $$ \sum_{p\leq T} f(p) \log p \ll \frac{T}{(\log T)^A}$$ for all $T\geq Q$ and with an implied constant that depends at most on 
$A$. Thus, by~\cite[Theorem 13.2, Remark 13.3]{koukou}
with $\kappa=0,\epsilon=1/4$ and $ k=1$ we obtain the bound $$\sum_{n\leq T} f(n) \ll \frac{T}{(\log T)^{2024}}
$$ as long as $\log T\geq (\log Q)^{5/4}$.
By assumption this is satisfied when $T=N$, thus concluding the proof.\end{proof} 
\begin{lemma}\label{lem:ptzno} Assume that $z=z(B)\to+\infty$ satisfies $  z \leq \log \log B$. Then 
for all $B$ satisfying $  \{ \log \c L(\mathrm e ^{3z}) \}^{5/4}\leq \log (B/(\log B)^{23} ) $
we have \[N(B;\c P)=  \sum_{\substack{ (\sigma ,\tau ) \in (\Z/W\Z)^2   \\ v_\ell(\sigma),v_\ell(\tau)\leq k_\ell \forall \ell\leq z   }}
(\c M'_{\sigma ,\tau }( B;z)+\c M''_{\sigma ,\tau }( B;z))
+O\l(\frac{1}{z^{1/2} (\log z) } \frac{B^2}{(\log B)^{\varpi} }+\frac{W^2 B^2}{(\log B)^{7/6}}\r), \]  
where  the sum  is subject to   $(\sigma ,\tau  )_{\Q_\ell}=1$ for all primes $ \ell \in \c P \cap[1,z] $ and 
$$  \c M'_{\sigma ,\tau }( B;z)=\sum_{\substack{  d, e \in \N }}  
\frac{\mu( d e)^2}{2^{\#\{ \ell\in \c P: \ell \mid d e \}} }, 
\ \ \ 
\c M''_{\sigma ,\tau }( B;z)=  \sum_{\substack{  d,e\in \N\\  \ell\mid de\Rightarrow \ell\in \c P   }}  
\frac{\mu(d e )^2}{\tau(de) } \l(\frac{b \tau'    }{d}\r) \l(\frac{a \sigma'     }{e}\r)
(-1)^{\frac{(d-1)(e-1)}{4}} ,$$ where both  $\c M',\c M''$ sums   are subject to 
 $  d   \equiv a\sigma/\sigma'    \md { W /\sigma'  } ,
 e  \equiv b\tau/\tau'    \md { W / \tau'   }$ as well as    
$  d \leq B/\sigma',   e  \leq B /\tau'  $.
The implied constant depends at most on $\c P$. \end{lemma}\begin{proof} By Lemma~\ref{saladthencoffee} it suffices 
to estimate $\c M_{\sigma,\tau}(B;z)$ and sum the error term over all $(\sigma,\tau) \in (\Z/W\Z)^2$.
The two sums in the present  lemma come respectively from 
the terms with $(d_1,e_1)=(1,1)$ and   $(d_2,e_2)=(1,1)$ in   Lemma~\ref{lem:fridiwanchilvert}.  It remains to deal with  the terms   terms satisfying 
$1<\max\{d_2,e_2\}\leq  (\log B)^{20}$. They contribute 
$$\ll \sum_{\substack{d_2,e_2 \leq (\log B)^{20}, \mu(d_2e_2)^2=1 \\ \gcd(d_2e_2,W)=1, (d_2,e_2) \neq (1,1)}} 
\bigg|    \sum_{\substack{ d_1 , e_1 \in \N,  \ell\mid d_1e_1\Rightarrow \ell\in \c P
  \\  d_1 d_2 \leq B/\sigma',  e_1e_2 \leq B /\tau' }}  
\frac{\mu(d_1)^2\mu(e_1)^2}{\tau(d_1)\tau(e_1)} \l(\frac{b \tau'  e_2 }{d_1}\r) \l(\frac{a \sigma'   d_2 }{e_1}\r)
(-1)^{\frac{(d_1-1)(e_1-1)}{4}}\bigg| ,$$ where the sum over $d_1,e_1$ is subject to $ \gcd(d_1 ,e_1 d_2)= \gcd( e_1, d_1 e_2)=1
$ and  $$ d_1   \equiv a\sigma/(d_2 \sigma'   ) \md { W /\sigma'  }
,e_1 \equiv b\tau/(e_2\tau' )   \md { W / \tau'   }  .$$With no loss of generality we may assume that $e_2\neq 1 $.
For fixed $d_2$, the value of   $d_1\md 8$  is fixed because $8 \mid W/\sigma'$.
If $2<p\leq z$ then the same congruence also fixes the value of the quadratic symbol  $(\frac{\tau'}{d_1})$ because $p$ divides   $W/\sigma'$.
Thus, the sum over $d_1$ becomes $$
 \sum_{\substack{ d_1  \in \N,  p\mid d_1 \Rightarrow p\in \c P
  \\  d_1  \leq B/(d_2\sigma' ),  \gcd(d_1,d_2e_1 )=1 \\
d_1   \equiv a\sigma/(d_2 \sigma'   ) \md { W /\sigma'  }
}}   \frac{\mu(d_1 )^2}{\tau(d_1 )} \l(\frac{  e_2 }{d_1}\r).$$ As above, the values of $e_2\md 4, d_1\md 4$ are fixed, thus, by quadratic reciprocity, 
we can replace $(\frac{  e_2 }{d_1})$ by $(\frac{d_1}{  e_2 })$ up to a constant sign. 
Our assumptions ensure that $W\leq \mathrm e^{3z}$. Indeed, as $z\to+\infty$ one has 
$$ \log W \leq 3\log 2+\log z +\sum_{p\leq z} (\log p+ \log z) \leq 3 z $$    by the Prime Number Theorem.
 By our assumptions $  \{ \log \c L(\mathrm e ^{3z}) \}^{5/4}\leq \log (B/(\log B)^{23} ) $ and $z\leq \log \log B$ 
we infer that $\{ \log \c L(\mathrm e ^{3z}) \}^{5/4}\leq \log (B/(d_2\sigma'))$, hence, the assumption $\exp(\{\log \c L(q)\}^{5/4} )\leq N$
of Lemma~\ref{lem:panolino} is met when $q=W/\sigma'$ and $N=B/(d_2\sigma')$. The integer $m=d_2e_1$ is at most $ B$, therefore, 
$$\omega(m) \ll \log m \ll \log B\ll  \log (B/(d_2\sigma'))\leq \exp(\sqrt{\log (B/(d_2\sigma'))}).$$ Taking $  \beta=e_2$ and using 
 that $\gcd(e_2,W)=1$ allows us to employ Lemma~\ref{lem:panolino} to infer that 
  the sum over $d_1$ is $\ll B/(d_2 (\log B)^{2024} )$. Thus, the overall contribution becomes 
$$\ll  \frac{ B}{  (\log B)^{2024} }\sum_{ d_2,e_2 \leq (\log B)^{20} }  \frac{1}{d_2} 
 \sum_{  e_1  \leq B /e_2}  1\ll  \frac{ B^2}{  (\log B)^{2000} },$$which is sufficient. \end{proof} 

We next show that if $\c P$ has density $\neq 1$ then $\c M''$ goes into the error term, while, 
if the density is $1$ one can simplify it by using Hilbert's product formula.
\begin{lemma}\label{lem:capricafedoesnottakecards}  Let $\sigma,\tau$ be as in  Lemma~\ref{saladthencoffee}.
If $\varpi \neq 1$ then
$\c M''_{\sigma ,\tau }( B;z)=O(B^2 (\log B)^{\varpi-2})$ with an absolute implied constant.
If $\varpi =1$ then        $$\c M''_{\sigma ,\tau }( B;z)=\c N''_{\sigma ,\tau }( B;z)
(a,b)_\R \prod_{\substack{  \ell\leq z  }}(\sigma,\tau)_{\Q_\ell} ,$$ where $$\c N''_{\sigma ,\tau }( B;z)
:=  \sum_{\substack{  d, e \in \N^2,  \ell\mid de\Rightarrow \ell\in \c P
  \\  d    \leq B/\sigma',  e   \leq B /\tau'  }}   \frac{\mu(d e )^2}{\tau(de) } $$          is subject to  $  d   \equiv a\sigma/\sigma'    \md { W /\sigma'  } ,
 e  \equiv b\tau/\tau'    \md { W / \tau'   }$. \end{lemma}\begin{proof}  Assume that $\varpi\neq 1$. 
The sum under consideration has modulus at most     $$   \bigg(  \sum_{\substack{ n \leq B \\  \ell\mid n \Rightarrow \ell\in \c P    }}  
\frac{\mu(n )^2}{\tau(n)  }   \bigg)^2 \ll \frac{B}{(\log B)^{2-\varpi}} $$ by~\cite[Theorem 1]{shiu} and~\eqref{AssumptionsonAAA}. 
In the remaining case $\varpi=1$ we use the  second assumptions in Theorem~\ref{thm:analytic} to see that $\c P$ contains all large enough primes. 
The sum under consideration comes from the terms $(d_2,e_2)=(1,1)$ 
in the sum of Lemma~\ref{lem:fridiwanchilvert}. These terms correspond to choosing the quadratic symbol for every prime $\ell$  in the expression
$$\prod_{\substack{\ell \mid s_1 \\ \ell\in \c P   }}\l( 1+\l(\frac{b\tau't_1}{\ell}\r) \r)
\prod_{\substack{\ell \mid t_1 \\ \ell \in \c P }}\l(1+\l(\frac{a\sigma's_1}{\ell}\r) \r)
=\prod_{\substack{\ell \mid s_1 \\ \ell\in \c P   }}\l( 1+\l(\frac{t}{\ell}\r) \r)
\prod_{\substack{\ell \mid t_1 \\ \ell \in \c P }}\l(1+\l(\frac{s}{\ell}\r) \r)
$$ in the notation of~\eqref{eq:hol2} of the proof of Lemma~\ref{lem:fridiwanchilvert}. Since 
the condition $\ell\mid s_1,\ell\in \c P$ is the same as 
$\ell \mid s, \ell>z, \ell\in \c P  $, the product becomes 
$$\prod_{\substack{\ell \mid s, \ell>z \\ \ell\in \c P   }} \l(\frac{t}{\ell}\r)  
\prod_{\substack{\ell \mid t, \ell>z \\ \ell \in \c P }} \l(\frac{s}{\ell}\r)  .$$ Using once again the fact that 
all    primes not in $\c P$   are bounded by $z$,  this turns into  
$$  \prod_{\substack{\ell \mid s \\ \ell>z  }} \l(\frac{t}{\ell}\r) 
\prod_{\substack{\ell \mid t \\ \ell >z }} \l(\frac{s}{\ell}\r) =\prod_{\substack{\ell \mid st \\ \ell>z  }}(s,t)_{\Q_\ell}=
\prod_{\substack{  \ell>z  }}(s,t)_{\Q_\ell}
 .$$ In the first equality we used that there are no   primes $>z$ dividing both $s,t$
and in the second we used that $(s,t)_{\Q_\ell}=1$ for all $\ell \nmid st$ with $\ell >z \geq 2$. 
Since $as,bt> 0 $ we infer that   $(s,t)_\R=(a,b)_\R$, hence, by Hilbert's product formula we obtain
$$ \prod_{\substack{  \ell >  z  }}(s,t)_{\Q_\ell}
=(s,t)_\R
\prod_{\substack{  \ell\leq z  }}(s,t)_{\Q_\ell}=(a,b)_\R\prod_{\substack{  \ell\leq z  }}(\sigma,\tau)_{\Q_\ell}
,$$ where we used $(s,t)\equiv (\sigma, \tau ) \md W$ and the periodicity of the Hilbert symbol in \S\ref{s:periodicity}.\end{proof} 
 \subsection{The end}\label{lem:joujou2} The next lemma deals with the characters that `correlate' with $\c P$.
\begin{lemma}\label{lem:blacksheepcoffee} Assume that $\varpi \neq 1 $ and let    $\psi$ be a non-principal Dirichlet 
character modulo $q$ with   $ c_\psi(1)  \neq 0 $. Then for all $m\in \N$  and all $T$ with    
  $\log T \geq \max\l \{(\log \c L(q))^{\frac{4}{1-\varpi}},\omega(m)^{1/10}\r\} $ we have  
$$\sum_{\substack{ n\leq T\\\gcd(n,m)=1  }} \frac{\mu(n)^2\psi(n)}{2^{\#\{\ell\in \c P: \ell \mid n \}}}
\ll \frac{T}{(\log T)^{1/2}},$$ where the implied constant depends at most on $\c P$.
\end{lemma}\begin{proof}Defining the multiplicative function $$ f(n) =  \frac{\mu(n)^2\psi(n)}{2^{\#\{\ell\in \c P: \ell \mid n \}}}\mathds 1(\gcd(n,m=1)$$
we have $$ \sum_{p\leq H} f(p) \log p =\sum_{p\leq H} \psi(p) +O(\omega(m) \log H)$$ with an absolute implied constant. Assume that 
$H\geq (\log T)^{20}$ so that our assumption on  $\omega(m)$ 
leads to  $\omega(m) \log H \ll H^{3/4}$. By the Siegel--Walfisz theorem~\eqref{AssumptionsonAAA} with $\beta=1$ we infer that if $H\geq \c L(q)$ then 
for all fixed $A>0$ this is $$ -\frac{c_\psi(1)}{2} H+O\l(\frac{H}{(\log H)^{A}}\r)$$  with an   implied constant depending at most on $A$ and $\c P$.
We now apply~\cite[Equation (13.11), Lemma 13.5 (a)]{koukou} with $$ Q=\max\l\{\c L(q),(\log T)^{20}\r\}, k=J=1, \kappa= -\frac{c_\psi(1)}{2}, \epsilon=
\frac{3+\varpi}{1-\varpi}$$ to infer that when  $\log T\geq (\log Q)^{1+\epsilon}$ then $$ \sum_{n\leq T} f(n) \ll 
\frac{T(\log Q)^2}{(\log T)^{1+\Re(c_\psi(1))/2 }},$$ with an implied constant that depends at most on 
 $\c P$ and $c_\psi(1)$. Since the $\psi$ with $c_\psi(1)\neq 0$ are determined only by $\c P$, the implied constant 
therefore depends only on  $\c P$. In light of   $\log Q \leq (\log T)^{1/(1+\epsilon)}$ we obtain the bound 
$O(T (\log T)^{-\lambda})$ where $\lambda= 1+\frac{\Re(c_\psi(1))}{2}-\frac{2}{1+\epsilon}$.
Using~\eqref{AssumptionsonAAA} twice yields $$ |c_\psi(1)| x +O\l(\frac{x}{\log x} \r) = \l|\sum_{\substack{ p\leq x \\ p\in \c P }} 
\psi(p)\left(\frac{p}{\beta}\right) \log p\r|\leq \sum_{\substack{ p\leq x \\ p\in \c P }}   \log p\leq \varpi x +O\l(\frac{x}{\log x} \r),$$ hence, 
$|\Re(c_\psi(1) )| \leq \varpi$. This means that $$\lambda \geq 1-\frac{\varpi}{2}-\frac{2}{1+\epsilon}   \geq \frac{1}{2} $$
due to  $ \epsilon  =  \frac{ 3+\varpi}{1-\varpi} $. The proof concludes by observing that $\log T \geq (\log Q)^{1+\epsilon}$ is satisfied 
due to our assumption  $\log T \geq (\log \c L(q))^{\frac{4}{1-\varpi}}$.\end{proof}

\begin{lemma}[Equidistribution inside progressions]\label{guitarbream} 
There exists $\theta>0$ that only depends on $\c P$ such that for all  $\sigma, \sigma', \tau,\tau'$ as in Lemma~\ref{lem:fridiwanchilvert}
and all $B $ satisfying    $$\log B \geq \max\l\{ \mathds 1 (\varpi \neq 1 )(\log \c L(\mathrm e^{3z} ))^{\frac{8}{1-\varpi}}, \mathrm e ^z\r\}$$
  we have  $$\c M'_{\sigma ,\tau }( B;z)= \frac{\mathfrak S_W}{\Gamma(1-\varpi/2)^2}\frac{B^2}{(\log B )^{\varpi}  } +O\l(\frac{B^2}{(\log B )^{\varpi+\theta}  } \r)
,$$ where the implied constant depends only on $\c P$. The constant $\mathfrak S_W$ is defined as$$\mathfrak S_W:= \frac{1}{\phi(W)^2 }
  \prod_{p} \l(1+\frac{\mathds 1(p>z )}{p2^{\mathds 1_{\c P}(p) }}\r)^2 \l(1-\frac{1}{p}\r)^{2-\varpi}
\hspace{-0.5cm}
\sum_{\substack{   \delta \in \N \\\gcd( \delta,W)=1  } }  \frac{\mu( \delta)\delta^{-2}}{ 4^{\#\{ \ell\in \c P: \ell \mid  \delta   \}} }
  \prod_{p\mid \delta} \l(1+\frac{1}{p2^{\mathds 1_{\c P}(p) }}\r)^{-2}.$$
\end{lemma}\begin{proof} Using M\"obius inversion to detect the coprimality of $d,e$ gives 
\beq{eq:ki8ara}{\sum_{\substack{   \eta \in \N , \eta \leq (\log B)^2\\\gcd( \eta,W)=1  } } \frac{\mu( \eta)}{4^{\#\{ \ell\in \c P: \ell \mid  \eta   \}} }
\sum_{\substack{    d \leq B/ (\eta\sigma'), \gcd(d, \eta)=1 \\ d   \equiv a\sigma/( \eta\sigma' )   \md { W /\sigma'  } 
 }}  \frac{\mu( d    )^2}{2^{\#\{ \ell\in \c P: \ell \mid d   \}} }  \sum_{\substack{     e  \leq B /( \eta\tau'), \gcd(e, \eta)=1
\\ e   \equiv b\tau/( \eta\tau')    \md { W / \tau'   }}}     \frac{\mu(  e  )^2}{2^{\#\{ \ell\in \c P: \ell \mid    e \}} }
 } up to an error term of size $\ll B^2 (\log B)^{-2}$. This is because  the contribution of $\eta> (\log B)^2$ is  
$$\ll\sum_{\eta>(\log B)^2}\frac{B^2}{\eta^2} \ll \frac{B^2}{(\log B)^2}.$$When $\eta\leq (\log B)^2$
we estimate the sum over $e$ by using orthogonality of characters to express it as 
\beq{eq:auriopetaei}{\frac{1}{\phi(W/\tau')}\sum_{\psi\md{W / \tau'  }}\overline{\psi(b\tau/( \eta\tau') )}\sum_{\substack{     e  \leq B /( \eta\tau')
\\  \gcd(e, \eta)=1  }}     \frac{\mu(  e  )^2\psi(e) }{2^{\#\{ \ell\in \c P: \ell \mid    e \}} }
.} We will  work in the   case    $\varpi \neq 1 $ for now and at the end we will describe the changes needed for the case $\varpi=1$.
Let us bound the contribution of non-principal characters in the sum over $\psi$. If $c_\psi(1) =0 $ then 
an argument similar to the one in Lemma~\ref{lem:ptzno} shows that the contribution is negligible.
If $\psi$ is non-principal and  $c_\psi(1) \neq 0 $ then the statements       $\varpi\neq 1$ and $\omega(\eta)\ll \log \eta \ll \log \log B$
allow us   to allude to Lemma~\ref{lem:blacksheepcoffee} to deduce that when 
  $\log B \geq (\log \c L(q))^{\frac{8}{1-\varpi}}$, one has 
$$ \sum_{\substack{     e  \leq B /( \eta\tau')
\\  \gcd(e, \eta)=1  }}     \frac{\mu(  e  )^2\psi(e) }{2^{\#\{ \ell\in \c P: \ell \mid    \delta \}} } \ll \frac{B}{\eta \sqrt{\log B}}$$ with an implied constant 
depending only on $\c P$. Thus, for $\eta$ as in~\eqref{eq:ki8ara} we have $$\sum_{\substack{     e  \leq B /( \eta\tau'), \gcd(e, \eta)=1
\\ e   \equiv b\tau/( \eta\tau')    \md { W / \tau'   }}}     \frac{\mu(  e  )^2}{2^{\#\{ \ell\in \c P: \ell \mid   e \}} }
=\frac{1}{\phi(W/\tau')}\sum_{\substack{     e  \leq B /( \eta\tau')
\\  \gcd(e, \eta W)=1  }}     \frac{\mu(  e  )^2  }{2^{\#\{ \ell\in \c P: \ell \mid  e \}} }
+O\l(  \frac{B}{\eta \sqrt{\log B}}\r)$$ with an implied constant 
depending only on $\c P$. We used  the definition of $\tau'$ to replace the information  that $e$ is coprime to $W/\tau'$ by $\gcd(e,W)=1$.
To estimate the sum over $e$ in the right-hand side we use the multiplicative function  $$ f(n)= 
\frac{\mu( n   )^2  }{2^{\#\{ \ell\in \c P: \ell \mid    n \}} }\mathds 1(\gcd(n, \eta W)=1  )   .$$ Using $W\leq \mathrm e^{3z}$ and the 
 properties  $\eta\leq (\log B)^2 $ and $z\leq \log \log B$ shows that 
$\omega(\eta W)\ll  \log \log B$ with an absolute implied constant. Hence, for $H\geq Q:= (\log \log B)^2$ we obtain 
$$ \sum_{p\leq H} f(p) \log p=\l(1-\frac{\varpi}{2} \r) H +O\l(\frac{H}{(\log H)^A}\r) $$ for any fixed $A>0$
with the implied constant depending only on $A$ and $\c P$. We now allude to~\cite[Equation (13.11)]{koukou} 
with $k=J=1$, $\kappa=1-\varpi/2$ and $\epsilon=3$. 
It shows that \beq{eq:elbetikasokolatakia}{\sum_{\substack{     e  \leq B /( \eta\tau')
\\  \gcd(e, \eta W)=1  }}     \frac{\mu(  e  )^2  }{2^{\#\{ \ell\in \c P: \ell \mid  e \}} }
=\frac{\widetilde c (\eta) }{ \eta\tau'} \frac{B }{(\log (B /( \eta\tau')))^{\varpi/2}}+O\l(\frac{B(\log Q)^2}{\eta(\log B)^{\varpi/2+1}}\r)
,} where $$ \widetilde{c}(\eta)= \prod_{p } \l(1+\frac{\mathds 1(p\nmid \eta W)}{p2^{\mathds 1_{\c P}(p) }}\r) \l(1-\frac{1}{p}\r)^{1-\varpi/2}
$$ and the implied constant depends at most on $\c P$. 
Note that $\gcd(\eta,W)=1 $ gives \beq{eq:xaxakes}{ \widetilde{c}(\eta)
\ll  \prod_{p\mid \eta W } \l(1+\frac{1}{p}\r)^{O_\varpi(1)}\ll (\log \log \eta )^{O_\varpi(1)}(\log z) ^{O_\varpi(1)},} where the implied constants
depend at most on $\c P$. Our assumption $   z \leq \log \log B$ ensures that $W\leq \mathrm e ^{3z} \leq (\log B)^3$ 
as in the proof of Lemma~\ref{lem:ptzno}, hence, 
$$(\log B/(\eta \tau'))^{-\varpi/2}=(\log B)^{-\varpi/2} \l(1+O\l(\frac{\log \log B}{\log B}\r) \r).$$ 
Hence,  using  $\tau' \phi(W/\tau')=  \phi(W)$ we   can  obtain   
\beq{eq:bluesmanblues}{\sum_{\substack{     e  \leq B /( \eta\tau'), \gcd(e, \eta)=1
\\ e   \equiv b\tau/( \eta\tau')    \md { W / \tau'   }}}     \frac{\mu(  e  )^2}{2^{\#\{ \ell\in \c P: \ell \mid   e \}} }
= \frac{\widetilde{c}(\eta) }{ \phi(W ) \eta} \frac{B }{(\log B)^{\varpi/2}}
+O\l(  \frac{(1+\widetilde c(\eta))  B}{ \eta\sqrt{ \log B}}\r).}The contribution of the error term towards~\eqref{eq:ki8ara} is 
$$\ll \frac{ B}{ \sqrt{ \log B}} \sum_{\substack{    \eta \leq (\log B)^2\\\gcd( \eta,W)=1  } }   \frac{(1+\widetilde c(\eta))   }{ \eta }
\sum_{\substack{    d \leq B/ (\eta\sigma'), \gcd(d, \eta)=1 \\ d   \equiv a\sigma/( \eta\sigma' )   \md { W /\sigma'  } 
 }}  \frac{\mu( d    )^2}{2^{\#\{ \ell\in \c P: \ell \mid d   \}} }  $$ which, by~\eqref{eq:bluesmanblues}, is at most 
$$\ll \frac{ B}{ \sqrt{ \log B}} \sum_{\substack{    \eta \leq (\log B)^2\\\gcd( \eta,W)=1  } }   \frac{(1+\widetilde c(\eta))   }{ \eta }
\l\{\frac{\widetilde{c}(\eta) }{ \phi(W ) \eta} \frac{B }{(\log B)^{\varpi/2}}
+   \frac{(1+\widetilde c(\eta))  B}{ \eta\sqrt{ \log B}} \r\}\ll    \frac{B^2 (\log z)^{O_\varpi(1)} }{  (\log B)^{(1+\varpi)/2}}. $$
This is acceptable because $z\leq \log \log B$ and $(1+\varpi)/2 >\varpi$.
 The main term in~\eqref{eq:bluesmanblues} contributes towards~\eqref{eq:ki8ara} the quantity 
$$  \frac{B }{\phi(W ) (\log B)^{\varpi/2}} 
\sum_{\substack{   \eta \in \N , \eta \leq (\log B)^2\\\gcd( \eta,W)=1  } } \frac{\mu( \eta)\widetilde{c}(\eta) }{\eta 4^{\#\{ \ell\in \c P: \ell \mid  \eta   \}} }
\sum_{\substack{    d \leq B/ (\eta\sigma'), \gcd(d, \eta)=1 \\ d   \equiv a\sigma/( \eta\sigma' )   \md { W /\sigma'  } 
 }}  \frac{\mu( d    )^2}{2^{\#\{ \ell\in \c P: \ell \mid d   \}} }  
,$$ which can be estimated by invoking~\eqref{eq:bluesmanblues} once again. The ensuing error term is 
$$\ll \frac{B ^2}{\phi(W ) (\log B)^{(1+\varpi)/2}} 
\sum_{\substack{   \eta \in \N , \eta \leq (\log B)^2\\\gcd( \eta,W)=1  } } \frac{(1+\widetilde c(\eta))^2 }{\eta^2  } 
\ll \frac{B ^2(\log z)^{O_\varpi(1)}}{(\log B)^{(1+\varpi)/2}}  .$$ The main term is 
$$\frac{B^2 }{\phi(W )^2 (\log B)^{\varpi }} 
\sum_{\substack{   \eta \in \N , \eta \leq (\log B)^2\\\gcd( \eta,W)=1  } } \frac{\mu( \eta)\widetilde{c}(\eta)^2 }{\eta^2 4^{\#\{ \ell\in \c P: \ell \mid  \eta   \}} }
.$$Completing the summation over $\eta$ produces an error term of size 
$$\ll \frac{B^2 }{ (\log B)^{\varpi }} 
\sum_{\substack{   \eta >(\log B)^2  } } \frac{\mu( \eta)\widetilde{c}(\eta)^2 }{\eta^2 4^{\#\{ \ell\in \c P: \ell \mid  \eta   \}} }
\ll \frac{B^2(\log z)^{O_\varpi(1)} }{ (\log B)^{2+\varpi }},
$$ which is acceptable. The resulting main term is then seen to be  
 $$\frac{B^2}{(\log B )^{\varpi}  } \frac{1}{\Gamma(1-\varpi/2)^2\phi(W)^2 }
\sum_{\substack{  \gcd( \delta,W)=1  } } 
\frac{\mu( \delta)}{ \delta^{2}4^{\#\{ \ell\in \c P: \ell \mid  \delta   \}} }
  \prod_{p} \l(1+\frac{\mathds 1(p\nmid \delta  W )}{p2^{\mathds 1_{\c P}(p) }}\r)^2 \l(1-\frac{1}{p}\r)^{2-\varpi}
, $$ that can be factored as stated in the lemma.

Lastly,   when $\varpi=1$ we know that all but finitely many primes are in    $\c P$, hence, by the Siegel--Walfisz theorem
one has $c_\psi(1)=0$ for all non-principal $\psi$. Hence, we only need to work with $\psi=1$ in~\eqref{eq:auriopetaei}. 
For this, one can follow similar steps as in our proof of~\eqref{eq:elbetikasokolatakia} to produce admissible error terms.
Note that the saving $1/\sqrt{\log B}$ in~\eqref{eq:bluesmanblues} would not be satisfactory, however, it comes only from the 
terms with $c_\psi(1)\neq 0$ and non-principal $\psi$ that are not relevant in the present case.\end{proof}

\begin{lemma}\label{eq:glenngould}For any fixed $A>0$ we have 
$$\mathfrak S_W=\frac{ 1+O((\log z)^{-A})}{W^2} \prod_{p\leq z }\l(1-\frac{1}{p}\r)^{-\varpi},
$$ where the implied constant depends at most on $A$ and $\c P$.\end{lemma}\begin{proof}
Each $\delta\neq 1 $ in the definition of $\mathfrak S_W$  exceeds $z$ since it is coprime to $W$.
Since each term in the sum over $\delta$ is $\ll \delta^{-2}$,  we obtain  
$$\sum_{\substack{   \delta \in \N \\\gcd( \delta,W)=1  } }  \frac{\mu( \delta)\delta^{-2}}{ 4^{\#\{ \ell\in \c P: \ell \mid  \delta   \}} }
  \prod_{p\mid \delta} \l(1+\frac{1}{p2^{\mathds 1_{\c P}(p) }}\r)^{-2}
=1+O\l(\sum_{\delta>z} \frac{1}{\delta^{2}}\r)=1+O(\delta^{-1}).$$  The terms $p>z$ in the product over $p$ of $\mathfrak S_W$ contribute
$$  \prod_{p>z} \l(1+\frac{1}{p2^{\mathds 1_{\c P}(p) }}\r)^2 \l(1-\frac{1}{p}\r)^{2-\varpi}=
\exp\l(\sum_{p>z} \frac{2^{1-\mathds 1 _\c P(p)}-2+\varpi}{p}+O\l(\frac{1}{p^2}\r) \r),$$ whose logarithm is 
$$ \sum_{p>z} \l(\frac{\varpi}{p}-\frac{\mathds 1_\c P(p)}{p}\r) +O\l(\frac{1}{z}\r)\ll \frac{1}{(\log z)^A}$$ by partial summation 
and~\eqref{AssumptionsonAAA}. \end{proof}

\begin{lemma}\label{lem:thanksforthecoffee} Assume that $\varpi=1$ and fix any $A>0$.
There exists $\varrho>0$ that only depends on $\c P$ such that for all 
$\sigma, \tau, \sigma',\tau',B,z$ as in Lemma~\ref{guitarbream},  we have 
$$\c N''_{\sigma ,\tau }( B;z)=\frac{1}{\Gamma(1/2)^2} \frac{ 1+O((\log z)^{-A})}{W^2} \prod_{p\leq z }\l(1-\frac{1}{p}\r)^{-1}
\frac{B^2}{\log B    } +O\l(\frac{B^2}{(\log B )^{1+\varrho}  } \r),$$
where the implied constants depends   only  on $A$ and $\c P$.
 \end{lemma}\begin{proof} The proof is similar to that of Lemmas~\ref{guitarbream}-\ref{eq:glenngould}. 
The situation is simpler here because when $\varpi=1$ the set $\c P$ contains all large enough primes by assumption, thus, 
$c_\psi(\beta)=0$ whenever  $\psi(\cdot) (\frac{\cdot}{\beta})$ is a non-principal Dirichlet character.\end{proof}

Recall the definition of $z_B$ in Theorem~\ref{thm:analytic} 
\begin{lemma}\label{eq:finalasymp} 
There exists $\varpi''>0$ such that if 
 $z=\min\{\varpi''\log \log B, z_B\}$ and  $ B\geq \c L(\mathrm e^{3z})$ then $$
\frac{N(B;\c P)}{B^2(\log B)^{-\varpi}}=\frac{1}{\Gamma(1-\varpi/2)^2}
 \frac{ 1+O((\log z)^{-A})}{ \prod_{p\leq z}(1-1/p)^{\varpi}}
(\c C+\mathds 1 (\varpi=1)\c C^*)
+O(z^{-1/2}),$$ where the implied constant depends only on $\c P$. Further, $$ \c C:=  W^{-2}  \sum_{\substack{ (\sigma ,\tau ) \in (\Z/W\Z)^2   
\\ v_\ell(\sigma),v_\ell(\tau)\leq k_\ell \forall \ell\leq z   }}1 \ \ \ \ \textrm{ and } \ \ \ \ \c C^*:=(a,b)_\R W^{-2}   
\sum_{\substack{ (\sigma ,\tau ) \in (\Z/W\Z)^2   
  \\ v_\ell(\sigma),v_\ell(\tau)\leq k_\ell \forall \ell\leq z   }}
\prod_{\substack{  p\leq z  }}(\sigma,\tau)_{\Q_p} $$  are both subject to 
$(\sigma ,\tau  )_{\Q_\ell}=1$ for all  $ \ell\leq z $ with $\ell \in \c P$.
\end{lemma} \begin{proof} Injecting   Lemma~\ref{lem:thanksforthecoffee}  into 
Lemma~\ref{lem:capricafedoesnottakecards}  yields an asymptotic for the  sum $\c M''_{\sigma ,\tau }( B;z)$ that can then be fed into
 in Lemma~\ref{lem:ptzno}. 
The first sum $\c M'_{\sigma ,\tau }( B;z)$  in Lemma~\ref{lem:ptzno}
can  be estimated by combining   Lemmas~\ref{guitarbream}-\ref{eq:glenngould}. Thus,   $\c M'_{\sigma,\tau}(B;z)+\c M''_{\sigma,\tau}(B;z)$ equals 
$$ \l(1+(a,b)_\R \prod_{\substack{  \ell\leq z  }}(\sigma,\tau)_{\Q_\ell} \r)
 \frac{   \{1+O((\log z)^{-P})\}  }{W^2\Gamma(1-\varpi/2)^2} \prod_{p\leq z }\l(1-\frac{1}{p}\r)^{-\varpi}
 \frac{B^2}{(\log B )^{\varpi}  }  +O\l(\frac{B^2}{(\log B )^{\varpi+\varpi'}  } \r)
,$$ where $\varpi'>0$ depends on $\c P$. Summing over $\sigma, \tau$ we then use    Lemma~\ref{saladthencoffee} 
to prove the desired asymptotic  for $B$ fulfilling   $W^2\sqrt z\leq (\log B)^{\varpi'}$,
$\log B \geq  \mathds 1 (\varpi \neq 1 )(\log \c L(\mathrm e^{3z} ))^{\frac{8}{1-\varpi}}$ and $ \log (B/(\log B)^{23} ) \geq \{ \log \c L(\mathrm e ^{3z}) \}^{5/4}$. 
Since $W\leq \mathrm e^{3z}$, the first inequality is satisfied if we further assume that $7z \leq \varpi' \log \log B$,
which is admissible by  taking $\varpi''$ to be a  a suitably small positive constant. For the last two inequalities to hold 
it is sufficient that $ B\geq \c L(\mathrm e^{3z})$.
\end{proof}

Recall that  $\mu_p= \mathrm{vol}(s,t\in \Z_p:(s,t)_{\Q_p}=1)$, where $\mathrm{vol}$ is the $p$-adic Haar measure.
\begin{lemma}\label{lem:mestenaxwreis}The equality   $\mu_2=13/18$ holds and for $p\neq 2 $ we have 
$$ \mu_p= \frac{2 p^2 + 2 p + 1}{2 p^2 + 4 p + 2}.$$ Furthermore, 
\beq{gabriellibrass}{\c C=\l(1+O\l(\frac{1}{\sqrt z} \r) \r )  \prod_{\substack{ p\leq z \\p\in \c P }} \mu_p.}
\end{lemma}\begin{proof}Define   $c_p=2\mathds 1_{\{2\}}(p)$ and 
\begin{align*}\sigma_p'&= \frac{\#\{(\sigma,\tau)\in (\Z/p^{k_p+1+c_p})^2: (\sigma,\tau)_{\Q_p}=1, v_p(\sigma),v_p(\tau)\leq k_p\}}{p^{2k_p+2+2c_p}}, \\
\tau_p'&= \frac{\#\{(\sigma,\tau)\in (\Z/p^{k_p+1+c_p})^2:   v_p(\sigma),v_p(\tau)\leq k_p\}}{p^{2k_p+2+2c_p}}.\end{align*}
A straightforward argument based on the Chinese remainder theorem shows that $$ \c C=\prod_{\substack{p \leq z\\ p\in \c P} }\sigma_p' 
\prod_{\substack{p \leq z\\ p\notin \c P} }\tau_p' .$$   With arguments similar to those in    the proof of Lemma~\ref{bachpartitas23}
we get  $$ \log \prod_{p\leq z } \frac{1}{1+O(p^{-1-k_p})} \ll \sum_{p\leq z } \frac{1}{p^{1+k_p}}  \ll \frac{1}{\sqrt z},$$ hence,  
$\prod_{p\leq z, p\in \c P }\tau'_p=1+O(z^{-1/2}) $. To study $\sigma'_p$ for $p\neq 2 $
we let  $\sigma=p^\alpha u , \tau=p^\beta v$ with $p\nmid uv$ to obtain 
\[p^{-2-2k_p} \sum_{\substack{ 0\leq \alpha, \beta \leq k_p }}
\#\l\{u \md{p^{1+k_p-\alpha}},  v \md{p^{1+k_p-\beta}}
:p\nmid uv , \l(\frac{v}{p}\r)^\alpha\l(\frac{u}{p}\r)^\beta =\l(\frac{-1}{p} \r)^{\alpha \beta }\r \}
.\] Since the last statement is periodic for $u,v \md p$ we see that this is 
\[p^{-2 }\sum_{\substack{ 0\leq \alpha,\beta \leq k_p }}p^{ -\alpha-\beta}
\#\l\{u,v \md{p}:p\nmid uv , \l(\frac{v}{p}\r)^\alpha\l(\frac{u}{p}\r)^\beta =\l(\frac{-1}{p} \r)^{\alpha \beta }\r \}
=p^{-2 }\sum_{\substack{ 0\leq \alpha , \beta \leq k_p }}
\frac{c(\alpha,\beta)}{p^{ \alpha+\beta}} ,\] say. Note that $0\leq c(\alpha,\beta)\leq p^2 $, hence, 
$$\sum_{\substack{  \alpha , \beta \geq 0 }}\frac{c(\alpha,\beta)}{p^{2+ \alpha+\beta}}-
\sum_{\substack{ 0\leq \alpha , \beta\leq k_p  }}
\frac{c(\alpha,\beta)}{p^{ 2+\alpha+\beta}} \ll  \sum_{\alpha>k_p} p^{-\alpha} \ll  1/z.
$$ Hence, the following estimate holds with an absolute implied constant, 
$$\sigma'_p=  \sum_{  \alpha , \beta \geq 0 } c(\alpha,\beta) p^{ -2-\alpha-\beta}+O\l( \frac{1}{z}\r)
.$$
Repeating the same arguments without the restrictions $v_p(\sigma),v_p(\tau)$ will give 
$$\mathrm{vol}(s,t\in \Z_p:(s,t)_{\Q_p}=1)= \sum_{  \alpha , \beta \geq 0 } c(\alpha,\beta) p^{ -2-\alpha-\beta},$$ so that 
$\sigma'_p= \mu_p+O(1/z)$.  Noting that  $ c(\alpha,\beta)$ depends on $\alpha, \beta\md 2 $ we can use the estimate 
$\sum_{\alpha \geq 0 , \alpha \equiv i \md 2 } p^{-\alpha}=p^{-i}(1-1/p^2)^{-1}$ for $i=0,1$ to  write \[\mu_p=p^{-2} (1-1/p^2)^{-2}\sum_{(\alpha,\beta)\in \{0,1\}^2} 
\frac{c(\alpha,\beta) }{p^{\alpha+\beta}}.\] Clearly,  $c(0,0)=(p-1)^2$ and $c(0,1)=c(1,0)=c(1,1)=\frac{ (p-1)^2 }{2} $, hence, 
\[\mu_p=\frac{(p-1)^2}{p^2(1-1/p^2)^{2} } \l(1+\frac{1}{p}+\frac{1}{2p^2} \r)
=\frac{2 p^2 + 2 p + 1}{2 p^2 + 4 p + 2}=1-\frac{1}{p}+O\l(\frac{1}{p^2}\r).\]
A similar argument for $p=2$ gives $\sigma'_2=\mu_2+O(1/z)$ and   \[ \mu_2= \frac{1}{6^2}  
\sum_{\substack{ ( \alpha, \beta)\in \{0,1\}^2     }} 2^{-\alpha-\beta}
\#\l\{u  , v \md{8}:2\nmid uv , \l(\frac{2}{v}\r)^\alpha\l(\frac{2}{u}\r)^\beta=(-1)^{\frac{(u-1)(v-1)}{4}}  \r \}
.\] When at most one $\alpha,\beta$ vanishes, the cardinality equals $12$; it equals $8$   if they are both $1$.
Thus the sum over $\alpha,\beta$ equals $26$ so that $\mu_2=13/18$.  
Since $\mu_p\gg 1$ with an absolute implied constant we obtain $\sigma'_p= \mu_p(1+O(1/z))$ from the bound $\sigma'_p= \mu_p+O(1/z)$, hence, 
$$ \prod_{\substack{ p\leq z \\ p\in \c P} } \sigma'_p=(1+O(1/z)) \prod_{\substack{ p\leq z \\ p\in \c P} } \mu_p.$$ Together with our earlier bound on 
-1+$\prod_{p\leq z,p\notin \c P}\tau'_p$ this proves~\eqref{gabriellibrass}.  \end{proof}

\begin{lemma}\label{estensendfiles}For $\varpi=1$ and $z=z(B)\to+\infty$ we have  
$$\c C^*=(a,b)_\R\l(1+O\l(\frac{1}{\sqrt z} \r) \r )  \prod_{\substack{ p\leq z \\p\in \c P }} \mu_p
\prod_{\substack{  p\notin \c P }} (2\mu_p-1).$$
\end{lemma}\begin{proof}Let $$\tau''_p = \frac{1}{p^{2k_p+2+2c_p}}
\sum_{\substack{(\sigma,\tau)\in (\Z/p^{k_p+1+c_p})^2\\  v_p(\sigma),v_p(\tau)\leq k_p}} (\sigma,\tau)_{\Q_p}.$$
A straightforward argument based on the Chinese remainder theorem shows that $$ \c C^*=(a,b)_\R
\prod_{\substack{p \leq z\\ p\in \c P} }\sigma_p' 
\prod_{\substack{p \leq z\\ p\notin \c P} }\tau''_p.$$ Noting that $\tau''_p=2\sigma'_p-\tau'_p$ and using the estimates for $\tau'_p, \sigma'_p$  
from the proof of Lemma~\ref{lem:mestenaxwreis} gives$$\c C^*=(a,b)_\R \l(1+O\l(\frac{1}{\sqrt z} \r) \r )  \prod_{\substack{ p\leq z \\p\in \c P }} \mu_p
\prod_{\substack{ p\leq z \\p\notin \c P }} (2\mu_p-1).$$ To conclude the proof, note that by assumption when $\varpi=1$ the set $\c P$ has all sufficiently large primes, thus, the condition $p\leq z $ is redundant in the second product since $z(B)\to+\infty$.
\end{proof} Next we give a lemma that is only needed later for the proof of Theorem~\ref{thm:perfectly}.
\begin{lemma}\label{lem:irrelevant}Fix an integer $n\geq 1$. For any $x_1,\ldots, x_n \in [0,1]$ we have 
$$2 \prod_{i=1}^n x_i \leq 1+\prod_{i=1}^n(2x_i-1).$$ If $n\geq 2 $ and 
there is $i$ with $x_i\neq 1$ then strict inequality holds. 
\end{lemma}\begin{proof}We use induction on $n$. Assume that it holds for $n$ and define $F:[0,1]\to \R$ by  
$$F(x)= 1+(2x-1)\prod_{i=1}^{n}(2x_i-1)-2 x\prod_{i=1}^{n} x_i .$$
The coefficient of $x$
is not exceeding $0$ because $2x_i-1 \leq x_i$ for each $i$. Hence, the minimum of $F(x)$  
occurs at $x=1$, thus, $$F(x_{n+1})\geq F(1)=1+\prod_{i=1}^{n}(2x_i-1)-2 \prod_{i=1}^{n} x_i,$$ which is in $[0,\infty)$   by the induction hypothesis.
This proves the first claim. The second claim is also proved by induction starting at $n=2$. The inductive step is the same as above.
To see why the claim holds for $n=2$ note that $2 x_1 x_2 < 1+ (2x_1-1)(2x_2-1)$ is equivalent to $ x_1+x_2<  1+x_1x_2$.
This in turn holds since $\min x_i<1 $ and $\max x_i \leq 1$. 
 \end{proof} 
\begin{remark}\label{rem:Leonardo_Leo_is_a_god}
Let $(a,b)\in \{1,-1\}^2\setminus \{(-1,1)\}$, let $\c P_1$ be the set of all primes
and $\c P_0$ be a set of primes  whose complement inside $\c P_1$
is finite. Then for the leading constant $c$ in Theorem~\ref{thm:analytic}
we have $ c(\c P_1) < c(\c P_0)$ when $\c P_1\setminus \c P_0$ contains at least  two distinct primes.
This can be seen either  by    taking $B\to\infty$ in Theorem~\ref{thm:analytic}
or from Lemma~\ref{lem:irrelevant} for $n=\#(\c P_1\setminus \c P_0)$. If $\c P_1\setminus \c P_0$ contains exactly one prime
then it can be inferred directly from the definition of $c$ that $c(\c P_0)=\c(P_1)$. This reflects Hilbert's reciprocity formula
as when a conic is soluble at all primes except one, then it must also be soluble at the missing prime, thus, $N(B,\c P_0)=N(B,\c P_1)$ in this case.
\end{remark}

\subsection{Proof of Theorem~\ref{thm:analytic}}Injecting Lemmas~\ref{lem:mestenaxwreis}-\ref{estensendfiles} into Lemma~\ref{eq:finalasymp} 
 shows that  $$\frac{N(B;\c P)}{B^2(\log B)^{-\varpi}}=\frac{1}{\Gamma(1-\varpi/2)^2}
 \frac{ \prod_{ p\leq z, p\in \c P  } \mu_p }{ \prod_{p\leq z}(1-1/p)^{\varpi}}
\l( 1+\mathds 1 (\varpi=1)(a,b)_\R \prod_{ p\notin \c P  } (2\mu_p-1) \r)
+O((\log z)^{-A}) ,$$  
where the implied constant depends only on $A$ and $\c P$.
By~\eqref{AssumptionsonAAA}  and partial summation one gets 
$$\prod_{p> z} \l(1 +\mathds 1_{\c P}(p) (\mu_p-1) \r) (1-1/p)^{-\varpi}
=1+O((\log z)^{-A}),$$ thus  $$  \frac{ \prod_{ p\leq z, p\in \c P  } \mu_p }{ \prod_{p\leq z}(1-1/p)^{\varpi}}
=O((\log z)^{-A})+\lim_{t\to\infty }\prod_{p\leq t} \frac{ 1  +\mathds 1_{\c P}(p) (\mu_p-1) }{ (1-1/p)^{-\varpi}}.$$
Since   $\log \min\{\log \log B, z_B\} \ll \log \min\{ \varpi''\log \log B, z_B\}$, we see that $N(B;\c P)B^{-2}(\log B)^{\varpi}$
is  $$\frac{c(\c P)}{\Gamma(1-\varpi/2)^2}
\l( 1+\mathds 1 (\varpi=1) (a,b)_\R\prod_{ p\notin \c P  } (2\mu_p-1) \r)
+O((\log \min\{\log \log B, z_B\})^{-A}) 
$$ with an implied constant that depends at most on $A$ and $\c P$. 
This concludes the proof of Theorem~\ref{thm:analytic}.

\section{Class field theory}
\label{s:diophantnstabl}

\subsection{Hilbert symbols in field extensions}
In this section $K$ denotes a {local number field}, namely a finite extension of $\Q_p$ for a prime number $p$. 
We denote by $v_K$ the {valuation} of the ring of integer $\mathcal{O}_K$ of $K$, normalized with the convention that $v_K(\pi_K)=1$,
for any $\pi_K$ that generates the maximal ideal of $\c O_K$.
The {Hilbert symbol} over $K$ is the function
$$(-,-)_K: \ \ \ 
K^{*}/K^{*2} \times K^{*}/K^{*2} \to \{1,-1\},
 $$ defined to be $(a,b)_K=1$, in case the conic $ax^2+by^2=z^2$ admits a non-trivial $K$-point. 
This symbol defines a symmetric non-degenerate bilinear pairing on $K^{\times}/K^{\times 2}$. 

We recall~\cite[Corollary 1.5.7]{MR2392026}      that describes how the Hilbert symbol changes under finite extensions. 
Many similar properties are established in~\cite[Section 14.2]{MR2392026}.
\begin{proposition} \label{HS: gets powered to degree}
Let $K$ be a local number field and let $a,b\in K^{\times}/K^{\times 2}$. For any finite extension  $M/K$ we have 
 $(a,b)_M=(a,b)_K^{[M:K]}$.\end{proposition}

We   recall the basic tame formula for the Hilbert symbol and  some mild control   in the wild case. For an 
element $a$ of $\mathcal{O}_K^{\times}$ we define
$(\frac{a}{\pi_K})_{K}$
to be $-1$ when $a$ is not a square in $\mathcal{O}_F/\pi_F$ and $1$ otherwise. 
We denote by $e_K$ the 
{ramification index} of the extension $K/\Q_2$. The following is standard:
\begin{proposition} \label{HS tame case+a little bit of wild}
$(a)$ Let $K$ be a local field of odd residue characteristic. Let $a\in \mathcal{O}_K^{\times}$ and $b\in K^{\times}$ such that 
 $2\nmid v_K(b)$. Then$$(a,b)_K=\Big(\frac{a}{\pi_K}\Big)_{\!K}.$$
$(b)$ Suppose that $K$ has residue characteristic equal to $2$. Let $a,b$ be two non-zero elements of $\mathcal{O}_K$. Then the value of $(a,b)_K$ is entirely determined by $a,b$ in $\mathcal{O}_K/\pi_K^{2 e_K+1}$. 
\end{proposition}

For the rest of this section $L$   denotes a number field and $L/F$     a finite extension. By Galois theory  $L$ corresponds to a transitive $G_F$-set $X_L$, for example, given by the set of roots of the minimal polynomial of a primitive element of $L$ over $K$.

The following result studies the change of Hilbert symbol under a field extension  in terms of the corresponding Galois set. 
Let $v$ be a finite place that is unramified in the Galois closure   
$N(L)/F$. We have a well-defined $\text{Gal}(N(L)/F)$-conjugacy class of elements denoted as 
$\text{Frob}_v$ in $\text{Gal}(N(L)/F)$, which we can view as permutations of $X_L$. Each of these permutations has the same cycle type
and there is a bijection between the cycles and the primes above $v$ in $F$: this bijection allows to read the local degree $[L_w:F_v]$ as the length of the corresponding cycle. If $v$ ramifies in $N(L)/F$, one can replace the conjugacy class of elements $\text{Frob}_v$ with the conjugacy class of subgroups given by the decomposition  groups at $v$ in $\text{Gal}(N(L)/F)$ and the cycles with the orbits of these decomposition groups. If $v$ is a non-archimedean prime, the  decomposition group at $v$ is still well-defined. 
\begin{proposition}\label{prop:tescoexpress}
Let $L/F$ be a finite extension of number fields and $a,b\in F^{\times}$. \\
$(a)$ For  a place $v$ of $L$ above a rational finite prime $w$ of $F$ we have  $(a,b)_{L_v}=(a,b)_{F_w}$ if  $v$ corresponds to an odd-length orbit of the decomposition group at $v$ acting on $X_L$. If the orbits are all even sized, then 
$(a,b)_{L_v}=1$ for all $v$ above $w$. \\
$(b)$ Let $v$ be a place of $L$ above a rational finite prime $w$ of $F$ that is unramified in the normal closure $N(L)$. Then
$(a,b)_{L_v}=(a,b)_{F_w}$
for all primes corresponding to an odd cycle of $\textup{Frob}_v$ in $X_L$. If $\textup{Frob}_v$ has only even cycles in $X_L$, then 
 $(a,b)_{L_v}=1$
for all $v$ above $w$. \\
$(c)$ For   a real archimedean place $v$ of $L$ we have  $(a,b)_{L_v}=(a,b)_{F_w}$. If $v$ is complex then $(a,b)_{L_v}=1$. \end{proposition}
\begin{proof}
Part $(c)$ is obvious. Parts $(a)$ and $(b)$ follow from Proposition \ref{HS: gets powered to degree}, once one recalls that for each prime $v$ 
above $w$,  $[L_{v}:F_w]$ is the size of the corresponding orbit of the decomposition group, which in the unramified case is cyclic and generated by 
the Frobenius element.\end{proof}

Let $G$ be a finite group and $X$   a transitive $G$-set. The elements $g\in G$ are viewed as a permutation on $X$.
Denote by $\textrm{Odd}_X(g)$ the number of   cycles of $g$ having odd length and  define 
\beq{Maramene, haggio visto 'na cosa}{\delta(G,X):=
\frac{\#\{g \in G: \textrm{Odd}_X(g)\neq 0\}}{\#G}.}
For a finite extension of number fields $L/F$  let $X_L$ be the corresponding  set and put 
 $$\delta_{L/F}:=\delta(\text{Gal}(N(L)/F),X_L).$$ We denote the places of $F$  as  $\Omega_F$ and let 
\beq{def:nationwide}{S_0(L/F)=\{v \in \Omega_F: \textrm{ decomposition groups at $v$ in $N(L)/F$ has only orbits of even size}\}. } 
The next result characterizes the finiteness of $S_0(L/F)$:
\begin{proposition} \label{Prop: S0 finite iff delta=1} The set $S_0(L/F)$ is finite if and only if $\delta_{L/F}=1$. Furthermore, 
in this case  $S_0(L/F)$ consists only of finite primes and all of its elements     ramify in $N(L)/F$. 
\begin{proof} If  $\delta_{L/F}=1$ then each element $g\in\text{Gal}(N(L)/F)$ 
 has only an odd cycle when it acts on $X_L$. Hence it suffices to observe that all of the decomposition groups at the infinite places 
and at the unramified finite primes are cyclic and hence  not in $S_0(L/F)$. This gives the first direction. 

Suppose now that $S_0(L/F)$ is finite. By Chebotarev's density theorem the Artin symbols of the primes that are unramified in $N(L)/F$ and    
 in the complement of $S_0(L/F)$  equidistribute in the set of conjugacy classes of $\text{Gal}(N(L)/F)$. Further,   each of these conjugacy classes
only has  elements whose cycle decomposition  never consists entirely of even cycles by the definition of $S_0(L/F)$.
It follows that this property holds with probability $1$ in $\text{Gal}(N(L)/F)$, in other words, $\delta_{L/F}=1$.
\end{proof}\end{proposition}  
\begin{proposition} \label{Prop: br cl as conic}
Let $F$ be a number field and  $\mathfrak{p},\mathfrak{q}$ be two distinct finite primes in  $\mathcal{O}_F$. Then there are infinitely 
many   $(a,b) \in (F^{\times}/F^{\times 2})^2$ such that $(a,b)_{F_{\mathfrak{p}}}=(a,b)_{F_{\mathfrak{q}}}=-1$, 
while for all other   $v\in\Omega_F$ we have    $(a,b)_{F_v}=1$.    \end{proposition} \begin{proof}
We claim that we can find a prime ideal $\mathfrak{l}_1$ of $M$ such that the ideal $\mathfrak{p}\cdot \mathfrak{l}_1$ is principal and admits a generator $\alpha$ with the following properties: \begin{enumerate} \item  $\alpha$ is a local square at every prime above $2$ different from $\mathfrak{p}$. In particular,
if $\mathfrak{p}$ is odd, we demand this at all such primes. \item  $\alpha$ is a local unit locally at $\mathfrak{q}$ such that $F_{\mathfrak{q}}(\sqrt{\alpha})/F_{\mathfrak{q}}$ is the quadratic unramified extension of $F_{\mathfrak{q}}$. \item  $\alpha$ is totally positive. \end{enumerate}

To prove this claim let     $\mathfrak{m}$   be modulus uniquely defined by 
demanding that every infinite place divides $\mathfrak{m}$, 
that for every place $w$ above $2 \mathfrak{q}$ and different from $\mathfrak{p}$
the ideal $w^{3 \cdot e_{L_w}}$ divides exactly $\mathfrak{m}$ and finally that no other place divides $\mathfrak{m}$. 

To this modulus corresponds the ray class group $\text{Cl}(F,\mathfrak{m})$; let $c$ be a class in this group. Observe that the class of $\mathfrak{p}$ in this ray class group is well defined, since $\mathfrak{p}$ does not divide $w$, by construction. Thanks to Chebotarev's density theorem, we can always find an odd prime ideal $\mathfrak{l}_1$ different from $\mathfrak{p},\mathfrak{q}$ such that $\mathfrak{p}\mathfrak{l}_1$ equals $c$. We recall that 
$\text{Cl}(F,\mathfrak{m})$ has   the subgroup $H/\mathcal{O}_F^{\times}$ coming from the principal ideals, where 
$$H=\prod_{\substack{ v|2 \mathfrak{q}\\ v \not \in \{\mathfrak{p}\}}}(\mathcal{O}_F/w^{3  e_{F_w}})^{\times} \times \{\pm 1\}^{v|\infty, \ \text{real}}
.$$
Specialize $c$ to be the $\mathcal{O}_{F}^{\times}$-coset of a class as prescribed in the proposition. 
Upon adjusting the result with a global unit, we obtain the claimed element $\alpha$. 

By   the same argument  we can   find  an odd prime ideal $\mathfrak{l}_2$ different from $\mathfrak{p},\mathfrak{q}$ 
and such that $\mathfrak{q}\mathfrak{l}_2=(\beta)$, that 
$\beta$ is a square also locally at $\mathfrak{l}_1$ and a  unit locally at $\mathfrak{p}$ such that $F_{\mathfrak{p}}(\sqrt{\beta})/F_{\mathfrak{p}}$ is the quadratic unramified extension of $F_{\mathfrak{p}}$. We have thus obtained $\alpha,\beta$ such that  the conic 
$$\alpha  X^2+\beta  Y^2=Z^2 $$
 is non-split at $\mathfrak{p},\mathfrak{q}$ 
as an element of    the $2$-torsion of the Brauer group of $F$. It is split at all infinite places as $\alpha$ is totally positive
and it  is split at $\mathfrak{l}_1$  because $\beta$ is locally a square. It is split at all places above $2$ different from $\mathfrak{p},\mathfrak{q}$ 
since  $\alpha$ is locally a square. Further, at all of the odd places coprime to $\mathfrak{l}_1\mathfrak{l}_2\mathfrak{p}\mathfrak{q}$
it  is also trivial, being the cup product of two unramified classes.
Summarizing, the conic $(\alpha,\beta)$ is locally trivial at all places except $\mathfrak{p},\mathfrak{q}$ and $\mathfrak{l}_1$ and it is non-trivial 
at the first two places. Hence by Hilbert reciprocity it has to be trivial also at $\mathfrak{l}_2$. 
Finally, observe that as we vary $\mathfrak{l}_1,\mathfrak{l}_2$ we get a set of $(\alpha,\beta)$ that is  
 linearly independent in $(\frac{F^{\times}}{F^{\times 2}})^2$, hence,
 it is infinite.\end{proof}

\begin{definition}
For a number field $F$ and  a finite extension $L/F$ we say that $L/F$ is \emph{stable in genus} $0$ when
 for each $a,b\in F^{\times}$ we have   $(a,b)_L=(a,b)_F$.    \end{definition} We next characterize stable extensions.
Denote the $2$-torsion of the Brauer group of a   field $F$ by $\textup{Br}(F)[2]$.
 \begin{theorem} \label{Thm: charact genus 0 stability}
Let $L/F$ be a finite extension of number fields. 
The following are equivalent:\\
$(a)$ The property $(a,b)_F=(a,b)_L$ holds for all  $a,b\in F^{\times}$. \\
$(b)$  $\# S_0(L/F) \leq 1$. \\
$(c)$ One has $\delta_{L/F}=1$ and   at most one  finite prime $v$   ramifying in $N(L)/F$ has a decomposition with only even sized orbits in $X_L$. \\
$(d)$ There are only finitely many elements $a,b$ in $F^{\times}/F^{\times 2}$ such that $(a,b)_F=-1$ and $(a,b)_L=1$. \\
$(e)$ The natural restriction map $\textup{Br}(F)[2] \to \textup{Br}(L)[2]$ is injective. \end{theorem}
\begin{proof}(b)$\Rightarrow$(e): Let $b$ be an element of $\text{Br}(F)[2]$, 
denote the restriction to $L$ as $\text{Res}_L(b)$ and assume that $\text{Res}_L(b)=0$. Then for all places $v$ of $L$ lying above a place $w$ of $F$, we 
have   $\text{Res}_{L_v}(b)=0$. Furthermore, note that $\text{Res}_{L_v}(b)=\text{Res}_{L_v}(\text{Res}_{F_w}(b))$.
By local class field theory we know that $\text{Res}_{F_w}(b)=(a_1,a_2)_{F_w}$ for   $a_1,a_2$ in $F_w^{\times}$. Then applying 
Proposition~\ref{HS: gets powered to degree} combined with the definition of $S_0(L/F)$ we conclude that $\text{Res}_{F_w}(b)$ vanishes for all $w$ outside 
  $S_0(L/F)$: 
indeed, Proposition \ref{HS: gets powered to degree} tells us that if $w$ is not in $S_0(L/F)$ then there is a place $v$ of $L$ above $w$ with 
$[L_v:F_w]$ odd, and thus $(a_1,a_2)_{F_w}=(a_1,a_2)_{L_v}=1$. Recall that     $S_0(L/F)$ has at most one element. By Hilbert reciprocity, we conclude
 that  $b$ restricts to $0$ locally at every place of $F$.
Therefore, by the local to global principle for the Brauer group, it follows that $b$ is $0$ as an element of $\text{Br}(F)$, giving the desired conclusion. 
The directions (e)$\Rightarrow$(a)$\Rightarrow$(d) are obvious.\\
(d)$\Rightarrow$(c): We proceed by contradiction. First suppose that $\delta(L/F)<1$. Then $S_0(L/F)$ is infinite thanks to 
Proposition~\ref{Prop: S0 finite iff delta=1}, thus,  we can find two  finite primes $\mathfrak{p},\mathfrak{q}$ in $S_0(L/F)$.  
Proposition~\ref{Prop: br cl as conic} produces infinitely many $a,b\in F^{\times}/F^{\times 2}$ such that $(a,b)_F$ is locally 
non-trivial precisely at $\mathfrak{p},\mathfrak{q}$ and nowhere else among the places of $F$. Therefore, combining the definition of $S_0(L/F)$ with 
Proposition \ref{HS: gets powered to degree}, we find   that $(a,b)_L$ vanishes at all primes above $\mathfrak{p},\mathfrak{q}$ and everywhere else. We 
have produced infinitely many pairs $(a,b)$ in $(F^{\times}/F^{\times 2})^{2}$ such that $(a,b)_F=-1$ but $(a,b)_L=1$. This is impossible if (d) holds. 
Therefore, we have shown by contradiction that   (d) implies   $\delta_{L/F}=1$. Furthermore, our argument has shown more generally that if (d) holds then 
$S_0(L/F)$ cannot contain two distinct finite primes of $M$. This proves that if $(d)$ holds then $(c)$ has to hold as well. \\
(c)$\Rightarrow$(b):  In view of Proposition~\ref{Prop: S0 finite iff delta=1},
$\delta_{L/F}=1$ implies that $S_0(L/F)$ consists only of finite 
places. But (c) prevents $S_0(L/F)$ from containing more than one finite prime. Therefore,  $\#S_0(L/F) \leq 1$, which concludes the proof.\end{proof}

The following corollary provides further information on $\delta_{L/F}$. 
\begin{corollary} \label{cor:lastoneok?}For any finite extension of number fields  $L/F$ we have   
$\delta_{L/F}>0$. Furthermore: \\ $(a)$ If $[L:F]$ is odd then $\delta_{L/F}=1$ and $(a,b)_L=(a,b)_F$ for all $a, b \in F^\times$. \\
$(b)$ If $L/F$ is Galois, then  $$\delta_{L/F}=\frac{\#\{g \in \textup{Gal}(L/F): 2 \nmid  \textup{ord}(g)\}}{\#\textup{Gal}(L/F)}.
$$In particular, if $L/F$ is Galois then 
 $2\nmid [L:F]\iff\delta_{L/F}<1$.\end{corollary}
\begin{proof}The fact that restriction composed with co-restriction, from $F$ to $L$, induces multiplication by $[L:F]$ in cohomology, shows that the 
restriction map $\text{Br}(F)[2] \to \text{Br}(L)[2]$ is injective when $[L:F]$ is odd. Therefore, by Theorem~\ref{Thm: charact genus 0 stability} part (a) holds.  
For part (b), note that since $L/F$ is Galois, the set $X_L$ has the regular $\text{Gal}(L/F)$-action. Hence, the length of the cycle of each element $g$ in 
$\text{Gal}(L/F)$   equals $\text{ord}(g)$, and thus the formula for $\delta_{L/F}$ follows. 
The final statement is then an immediate consequence of the fact that every group of even order admits an element of order   $2$, which is a 
special case of a well-known theorem of Cauchy.\end{proof}

\subsection{Uniform Chebotarev error terms} 
Given a subset of the primes $\c A$, can we characterize   the Dirichlet characters $\chi$
for which the average of  $\chi(p)$ exhibits cancellation as $p$ ranges over $\c A$? In this subsection we     use 
arguments from class field theory to answer this for certain `algebraic' $\c A$. 
Furthermore, we shall give uniform error terms by using work of 
Thorner and Zaman~\cite{thorn}. For this it is necessary to
prove    discriminant bounds; these are given in   Propositions~\ref{prop: disc under extensions}-\ref{prop: upper bound on disc}.

Let $M_2/M_1$ be a finite extension of number fields of degree $n$. We denote by $\text{Disc}(M_2/M_1)$ the discriminant ideal of $M_2$ over $M_1$. This is the $\c O_{M_1}$-ideal generated by $\text{Disc}(e_1, \ldots,e_n)$, as $\{e_1, \ldots, e_n\}$ runs over $n$-sets in $\mathcal{O}_{M_2}$ and where $\text{Disc}(e_1, \ldots,e_n)$ is the determinant of the Gram matrix whose $(i,j)$-th entry is $\langle e_i, e_j \rangle=\text{Tr}_{M_2/M_1}(e_ie_j)$.

The following basic property can be found in~\cite[Chapter III, Proposition 8]{Serre:loc fields}. 
\begin{proposition}\label{prop: disc under extensions}
Let $M_3 \supseteq M_2 \supseteq M_1$ be finite extensions of number fields. Then 
$$\textup{Disc}(M_3/M_1)=\textup{Disc}(M_2/M_1)^{[M_3:M_2]}   N_{M_2/M_1}(\textup{Disc}(M_3/M_2)). 
$$ \end{proposition}The next proposition gives control on the discriminant of a compositum of extensions. 
\begin{proposition} \label{Prop: disc for linearly disjoint}
Let $L_1,L_2$ be number fields inside a given separable closure of $\Q$, both containing a common number field $M$ and with $[L_1L_2:M]=[L_1:M][L_2:M]$. Then $\textup{Disc}(L_1L_2/M)$ divides  $$\textup{Disc}(L_1/M)^{[L_2:M]} \textup{Disc}(L_2/M)^{[L_1:M]}. $$\end{proposition}
\begin{proof} As special sets of size $[L_1L_2:M]=[L_1:M]\cdot [L_2:M]$, we can pick
 product sets of a choice of a $[L_1:M]$-set in $L_1$ and a $[L_2:M]$-set in $L_2$. The resulting Gram matrix is the tensor product of the two respective 
matrices. Further, for any two matrices $A,B$ one has
$$\text{det}(A \otimes B)=\text{det}(A)^{\text{ord}(B)}  \text{det}(B)^{\text{ord}(A)},$$
where $\text{ord}(-)$ is the function that sends a $j \times j$ matrix to $j$. We conclude that every element of the ideal 
$\textup{Disc}(L_1/M)^{[L_2:M]} \textup{Disc}(L_2/M)^{[L_1:M]}$ is inside the ideal $\textup{Disc}(L_1L_2/M)$, hence, the latter divides the former. \end{proof}

We conclude with a rough upper bound for the discriminant of a compositum.
\begin{proposition} \label{prop: upper bound on disc} 
For any Galois number fields $L_1,L_2$        inside a separable closure of $\Q$ we have 
$$|\textup{Disc}(L_1L_2/\Q)| \leq \frac{|\textup{Disc}(L_1/\Q)|^{[L_2:L_1\cap L_2]}  
|\textup{Disc}(L_2/\Q)|^{[L_1:L_1\cap L_2]}}{|\textup{Disc}(L_1\cap L_2/\Q)|^{ [L_1:L_1\cap L_2][L_2:L_1\cap L_2] }} .$$
\end{proposition} \begin{proof}  Let $M=L_1\cap L_2$. By Proposition \ref{prop: disc under extensions} with  $M_3=L_1L_2, M_2= M, M_1= \Q$ we get  
$$|\text{Disc}(L_1L_2/\Q)|=|\text{Disc}(M/\Q)|^{[L_1L_2:M]} \  |N_{M/\Q}(\text{Disc}(L_1L_2/M))|.$$ We bound $N_{M/\Q}$
by using Proposition \ref{Prop: disc for linearly disjoint} and  $[L_1L_2:M]=[L_1:M][L_2:M]$. Thus,  
 $$|\text{Disc}(L_1L_2/\Q)| \leq \xi_1^{[L_2:M]}\xi_2^{[L_1:M]} ,$$ 
where $ \xi_1 =|\text{Disc}(M/\Q)|^{[L_1:M]}  |N_{M/\Q}(\text{Disc}(L_1/M))|$
and $\xi_2=  |N_{M/\Q}(\text{Disc}(L_2/M))|$. Using Proposition \ref{prop: disc under extensions} with  $M_3=L_1, M_2= M, M_1= \Q$
we obtain  $\xi_1 = |\text{Disc}(L_1/\Q)|$. Furthermore, 
$$\xi_2=  \frac{|\text{Disc}(L_2/\Q)| }{|\text{Disc}(M/\Q)|^{ [L_2:M]}}
$$ by    Proposition \ref{prop: disc under extensions} with  $M_3=L_2, M_2= M, M_1=\Q$. 
\end{proof}

Every primitive Dirichlet character $\chi\md n$  has an associated field given  in~\cite[pg.21]{wash} as   
$$F(\chi):= \{a\in \Q(\zeta_n): ga=a \ \forall g\in \mathrm{Ker}(\chi)\} ,$$
where $\zeta_n$ denotes a $n$-th root of unity and 
the kernel is defined by  viewing $\chi$ as a character of the cyclotomic field $\Q(\zeta_n)/\Q$. 
Assume that we are given a finite Galois extension
$k/\Q$  and   a     union $S$ of conjugacy classes of $\mathrm{Gal}(k/\Q)$. 
Lastly, assume that $\c A$ is a subset of primes with the property that 
every  $p$ that  is unramified in $k/\Q$ is in   $  \c A$
if and only if $\mathrm{Frob}(p,k/\Q)\in S$. 

Next, we introduce the constant $m(\chi,k,S)$ that turn to     be  the mean of  $\chi(p)$ as $p$ ranges over $\c A$.
Define the compositum $E=k\cdot F(\chi)$ and note that $E/\Q$ is Galois since both $k$ and $F$ are.
The group $\mathrm{Gal}(E/\Q)$ embeds as a subgroup of the direct product $\mathrm{Gal}(E/\Q) \times \mathrm{Im}(\chi)$, with both projections being surjective. Hence, every conjugacy class $C$ of $E/\Q$ can be written uniquely as $c\times \{\lambda\}$ for some
 uniquely defined conjugacy class $c$ of $\mathrm{Gal}(k/\Q)$ and  $\lambda\in \mathrm{Im}(\chi)$ because $ \mathrm{Im}(\chi)$ is abelian.
Considering  the first coordinate projection provides us with a  well-defined surjective map 
 $$\pi:\{\text{conjugacy classes } C \text{ of} \ \text{Gal}(E/\Q) \} \to \{\text{conjugacy classes } c \text{ of} \  \text{Gal}(k/\Q) \}.$$
Similarly, the second coordinate projection gives   a well-defined surjective map 
$$\{\text{conjugacy classes } C \text{ of} \ \text{Gal}(E/\Q) \} \to \text{Im}(\chi),
$$which  we denote as $C \mapsto \chi(C)$. 
We can then define $$m(\chi,k,S) :=\frac{1}{\#\text{Gal}(E/\Q)}
\sum_{\substack{C: \ \pi(C) \in S}} \chi(C) \#C ,$$ 
where the sum is over   conjugacy classes $C$ of $\mathrm{Gal}(E/\Q)$ such that $\pi(C) \in S$.
\begin{lemma}\label{lem:scarlatti1}For $\chi, k,S,\c A$ as above we have 
$$\lim_{x\to\infty} \frac{1}{\#\{\mathrm{prime } \ p\leq x\}} \sum_{\substack{p\mathrm{ \ unramified \ in \  } k\\p\leq x, p\in \c A }}\chi(p)=
m(\chi,k,S).$$\end{lemma}
\begin{proof}For a prime $p$ that unramified in $E/\Q$ we  have
 $$\chi(p)=\chi(\text{Frob}(p,E/\Q))\textrm{ and }\text{Frob}(p,k/\Q)=\pi(\text{Frob}(p,E/\Q)).$$ 
Hence, by the definition of $\c A$ we have $$ \sum_{\substack{p\leq x  \\p\in \c A }}\chi(p)
=\sum_{C: \pi(C)\in S} \sum_{\substack{p\leq x  \\ \text{Frob}(p,E/\Q) \in C }}\chi(p)
+O(1)=\sum_{C: \pi(C)\in S} \chi(C) \sum_{\substack{p\leq x  \\ \text{Frob}(p,E/\Q) \in C }}1
+O(1),$$ where   $O(1)$    takes into account the ramified primes and  
     $C$ runs over conjugacy classes of $\mathrm{Gal}(E/\Q)$. 
By Chebotarev's density theorem we then obtain $$\sum_{C: \pi(C)\in S} \chi(C) \frac{\#C\#\{p\leq x\}}{\#\mathrm{Gal}(E/\Q)}
+o(\#\{p\leq x\}).$$  Dividing by the number of primes up to $x$ concludes the proof.
\end{proof}

\begin{remark}\label{rem:wellknown}Let $F$ be a field, let $F^{\text{sep}}$ be a separable closure
and let $L_1,L_2$ be     Galois extensions of $F$ that are both contained in  $F^{\text{sep}}$.
The  fibered product of the two Galois groups over the intersection is denoted   
$\text{Gal}(L_1/F) \times_{\text{Gal}(L_1 \cap L_2/F)} \text{Gal}(L_2/F)$ and defined     as
$$\{(g_1,g_2) \in \text{Gal}(L_1/F) \times \text{Gal}(L_2/F): {g_1}_{|L_1 \cap L_2}={g_2}_{L_1 \cap L_2} \}.
$$ Let us see why the    map  $g \mapsto (g_{|L_1},g_{|L_2})$ gives  a natural identification 
$$\text{Gal}(L_1 L_2/F) \simeq \text{Gal}(L_1/F) \times_{\text{Gal}(L_1 \cap L_2/F)} \text{Gal}(L_2/F).$$
Indeed, note that each $g$ acts identically on the intersection, irrespective of whether it is first restricted from
$L_1$ or from $L_2$. Consequently, we deduce that the image lies within $$\text{Gal}(L_1/F) \times_{\text{Gal}(L_1 \cap L_2/F)} \text{Gal}(L_2/F).$$
To prove injectivity of the map, we use the fact that any automorphism extends to further Galois extensions, therefore, 
the claim is   reduced to the case of the direct product, which is straightforward.\end{remark}
For a Galois extension $k/\Q$ we denote by $T_k$  the finite group of Dirichlet characters coming from $k$, that is those 
characters corresponding to cyclic extensions of $\Q$ sitting inside $k$.
Alternatively, $T_k$ is the finite group  of Dirichlet characters with $F(\chi) \subseteq k$. Note that       $\#T_k \leq \#\text{Gal}(k/\Q)^{\text{ab}}$, where 
$G^{\text{ab}}$ denotes the abelianization of a group $G$.

\begin{lemma}\label{lem:scarlatti2}Fix $k,S,\c A$ as above. 
For each primitive Dirichlet       $\chi \notin T_k$ we have    $ m(\chi,k,S)=0$.\end{lemma}\begin{proof} 
 We claim that for each $c$ in $S$ we have  
$$\sum_{C: \ \pi(C)=c} \chi(C) \#C=0.$$
As argued above, if $\pi(C)=c$ then the  conjugacy class $C$ has always the shape  $c\times \{\lambda\}$, thus,
  $\#C=\#c$. Hence, the claim is equivalent to stating that for each $c$ in $S$ one has 
\beq{eq:claimleg}{\sum_{C: \ \pi(C)=c} \chi(C) =0.}
Fix an element $g$ of $\text{Gal}(k/\Q)$. We claim that the set of   roots of unity $\lambda$ such that$$(g,\lambda) \in \text{Gal}(E/\Q) \subseteq \text{Gal}(k/\Q) \times \text{Im}(\chi)$$
consists of a non-empty collection of cosets by a \emph{non-trivial} subgroup of $\text{Im}(\chi)$ as soon as $\chi$ is not in $T_k$. It is non-empty because the first coordinate map is surjective. To see this we use  Remark~\ref{rem:wellknown}
with  $F=\Q, L_1=k, L_2=F(\chi)$ to get  the   identification
$$\text{Gal}(L/\Q) \to \text{Gal}(k/\Q) \times_{\text{Gal}( F(\chi) \cap k/\Q)} \text{Gal}(F(\chi)/\Q).$$
Since $\chi\notin T_k$ and $\text{Gal}(F(\chi)/\Q)=\text{Im}(\chi)$, we deduce  that there is $t>1$ dividing $\text{deg}(\chi)=\#\text{Im}(\chi)$ such that $\text{Gal}( F(\chi)\cap k /\Q)=\text{Im}(\chi^{t})$, as well as 
$$\text{Gal}(E/\Q) \simeq
\{(g,\lambda) \in \text{Gal}(k/\Q) \times \text{Im}(\chi): \chi^{t}(g)=\lambda^{t}\}.
$$
Since $t$   divides $\#\text{Im}(\chi)$ and $t>1$, the group
 $\text{Im}(\chi)[t]=\{\lambda \in \mathbb C: \lambda^{t}=1\}$ has $t>1$ elements. Hence, for each   $g\in\text{Gal}(k/\Q)$ the set of $\lambda$ that appear as coordinate of $(g,\lambda)$ in $\text{Gal}(E/\Q)$ form a coset of the group of $t$-th roots of unity. It follows that the terms in~\eqref{eq:claimleg}
can be arranged in blocks of cosets under the $t$-th roots of unity. Such a coset is   a 
 regular polygon on the unit circle and hence has $0$ as its center of mass. This proves~\eqref{eq:claimleg} and thus concludes the proof.
\end{proof}

\begin{lemma}\label{lem:scarlatti3}Fix $ \c A, k$ and $S$ as above. If $\c A$ has natural density $1$ among the primes 
then $\c A$ contains   all but finitely many primes.
\end{lemma}\begin{proof}By Chebotarev density theorem, we must have that the conjugacy classes of $S$ have total mass $1$ in $\text{Gal}(k/\Q)$. Since this is a finite probability space, this is the same thing as saying that $S$ consists of {all} equivalence classes of this group. Therefore the set of exceptional primes in this case is precisely the set of ramified primes in $k/\Q$, which is   finite.\end{proof}

We give  a quantitative version of Lemma~\ref{lem:scarlatti1} using bounds for  Landau--Siegel zeros
and a special case of recent work of Thorner--Zaman~\cite{thorn} 
on Chebotarev's density theorem for number fields that do not contain many quadratic subfields. 
Recall the standard  result~\cite[Lemma 3]{stark} that for a finite Galois extension $M/\Q$
the Dedekind zeta function $\zeta_M$  has at most one real zero  in the interval $[1-1/(4 \log |\textrm{Disc}(M)|), 1)$. If such  a zero exists, it
  is called the Landau--Siegel zero of $\zeta_M$.
\begin{lemma}[Thorner--Zaman]\label{lem:scarlatti4}
Fix  any positive constants  $A$ and $N$.
There exist   positive absolute  constants        $\gamma, \gamma_1$ such that
for any   finite Galois extension   $M/\Q$  with Galois group $G$,  
any conjugacy class $C\subset G$ 
and any $x \geq (|\mathrm{Disc}(M/\Q)| [M:\Q]^{[M:\Q]})^{\gamma_1}$
for which all quadratic extensions $M_0/\Q$ contained in $M$ satisfy $|\mathrm{Disc}(M_0/\Q)| \leq (\log x)^N$, 
 we have that the number of, unramified in $M$,  primes      whose Artin symbol is in $C$
and with     $|\mathrm N_{M/\Q}(p)|\leq x$ equals  
$$\frac{\#C}{\#G} \int_2^x\frac{\mathrm dt }{\log t } \left(1+O((\log x)^{-A})\right),$$ where the implied constant is independent of $x,C$ and $M$.
\end{lemma}\begin{proof} Let  $n_M:=[M:\Q]$ and $D_M:=|\mathrm{Disc}(M/\Q)|$.
By~\cite[Theorem 1.1]{thorn}   there are absolute constants $\gamma_2,\gamma_3>0$ such that 
for  $x \geq (|D_M| n_M^{n_M})^{\gamma_2}$ 
  the cardinality equals  $$ \frac{\#C}{\#G} \left( \int_2^x\frac{\mathrm dt }{\log t } +O(x^{\beta_1} )\right)\left(1+O\l(
\exp\l[-\frac{\gamma_3\log x}{\log(|D_M|n_M^{n_M} )}\r]
+ \exp\l[-\frac{(\gamma_3\log x)^{1/2} }{n_M^{1/2} } \r]
\r)\right),$$ where we used   the trivial inequality $\mathrm{Li}(x)\ll x $ and   $\beta_1\in(0,1)$ is a possible Landau--Siegel zero of the Dedekind zeta function $
\zeta_M$. If  $\gamma_1=\max\{\gamma_2,4/(\gamma_3\log 2)\}$
and  $x \geq (|D_M| n_M^{n_M})^{\gamma_1}$  then we see that $2^{n_M \gamma_1} \leq x$, hence
$n_M \leq \frac{\gamma_3}{4}  \log x $. This makes the  second error term be  
 $$ \ll  \exp\l[-\frac{(\gamma_3\log x)^{1/2} }{2} \r] \ll_A(\log x)^{-A}.$$ Similarly, since $\gamma_1 \geq 2/\gamma_3$
we get  $ \log ((|D_M| n_M^{n_M})) \leq \frac{1}{2} \log x $ 
 from   $x \geq (|D_M| n_M^{n_M})^{\gamma_1}$.  In particular, 
$$ \exp\l[-\frac{\gamma_3\log x}{\log(|D_M|n_M^{n_M} )}\r] \ll_A(\log x)^{-A}. $$
We next deal with the error term $O(x^{\beta_1})$. By Heilbronn's theorem~\cite{heilbronn} there is a quadratic 
extension $M_0/\Q$ contained in $M$ and whose Dedekind zeta function 
vanishes  at $\beta_0$.  By Siegel's well-known work~\cite[Theorem 12.10]{koukou} we know that for every $\epsilon>0$ there exists 
an ineffective constant $c(\epsilon)>0$ such that $\beta_1\leq 1- c(\epsilon) |\textrm{Disc}(M_0/\Q)|^{-\epsilon}$, therefore, 
$x^{\beta_1} \leq x \exp (-c(\epsilon) (\log x ) |\textrm{Disc}(M_0/\Q)|^{-\epsilon} )$. Recalling the assumption  of the present lemma 
$|\mathrm{Disc}(M_0/\Q)| \leq (\log x)^N$  and taking $\epsilon =1/(2N)$ shows that 
$$x^{\beta_1} \leq x \exp \l(-c(1/(2N)) \sqrt{\log x }\r)\ll_A \textrm{Li(x)} (\log x)^{-A}.$$
Taking $\gamma=\gamma_3$ concludes the proof.  \end{proof}

\begin{lemma}\label{lem:scarlatti5} Fix any $A,N>3$ and let $k, S,\c A$  be as above. For any 
square-free integer  $\beta$, any Dirichlet character $\psi$ of conductor $q$ coprime to $\beta$
and any   $ q\leq (\log x)^{9/10}, |\beta|\leq (\log x)^{N/3}$ we have 
  $$ \frac{1}{x}\sum_{\substack{p\mathrm{ \ unramified \ in \  } k\\p\leq x, p\in \c A }}
\l(\frac{p}{\beta}\r)\psi(p)\log p=m(\chi,k,S) +O_A\left(\frac{1}{(\log x)^A}\right),$$
where the implied constants depend at most on $A,N$ and $k$.
\end{lemma}
\begin{proof} The sum can be written as 
$$ \sum_{\substack{p\leq x, p\nmid \beta q  \\ p\in \c A }} \l(\frac{p}{\beta}\r)\psi(p)\log p+O(\#\{p\mid \beta q\} \log x).$$ 
A prime $p\nmid \beta q\textrm{Disc}(k/\Q)$  
  is  unramified in the compositum $E=F((\frac{\cdot}{\beta}))  F(\psi)  k$, thus, the sum 
  equals $$ \sum_{\substack{p\in \c A\cap [2,x]   \\  p \textrm{ unramified in }E/\Q}} \l(\frac{p}{\beta}\r) \psi(p)\log p+O(\#\{p\mid \beta q
\textrm{Disc}(k/\Q)  
\}  \log x).$$ 
As in Lemma~\ref{lem:scarlatti1} with $\chi(\cdot)=(\frac{\cdot}{\beta}) \psi(\cdot)$
we can write this as  
$$\sum_{C: \pi(C)\in S} \l(\frac{C}{\beta}\r)\psi(C) \sum_{\substack{p\leq x, \text{Frob}(p,E/\Q) \in C\\  p \textrm{ unramified in }E/\Q }} \log p
+O(\#\{p\mid \beta q \textrm{Disc}(k/\Q) \} \log x),$$ where  the first sum is over   conjugacy classes   $C$      of $\mathrm{Gal}(E/\Q)$. 
The error term is trivially bounded by $O_k((\log x)^{1+N})$ by our assumption on the size of $q$ and $|\beta|$.
The proof is now completed by invoking Lemma~\ref{lem:scarlatti4} 
for   $M=E$  together with partial summation to deal with the factor $\log p$.
This gives an asymptotic for $x$ under certain assumption on the growth of $x$ 
that we verify in the remaining of the proof.

Let $M_0/\Q$ be a quadratic subextension of $E$ and note   that $E$ ramifies at the prime divisors of 
$$\text{Disc}(k/\Q)   \text{Disc}(F(\psi)/\Q)   \text{Disc}(F((\cdot/\beta))/\Q),  
$$ hence $\text{Disc}(M_0/\Q)$ divides 
$8  \text{Disc}(k/\Q)  \text{rad}(\text{Disc}(F(\psi)/\Q))  \text{Disc}(F((\cdot/\beta))/\Q),$
where $\text{rad}$ denotes the radical. The extension $F(\psi)/\Q$ is a subextension of $\Q(\zeta_{q})/\Q$, hence, it ramifies only at 
the prime divisors of $q$. Thus, $\text{rad}(\text{Disc}(F(\psi)/\Q))$ divides $q$ and noting that $\text{Disc}(F((\cdot/\beta))/\Q)\ll |\beta|$
we conclude that $|\text{Disc}(M_0/\Q)| \leq c_k q |\beta|$, where $c_k$ depends only only the field $k$. For $x$ large enough
compared to $c_k$ and $N$ we can  see that  $|\text{Disc}(M_0/\Q)| \leq (\log x)^N$ by using 
the assumption that both   $ q$ and $ |\beta|$ are bounded by $ (\log x)^{N/3}$.

It remains to verify the condition $x \geq (|\mathrm{Disc}(E/\Q)| [E:\Q]^{[E:\Q]})^{\gamma_1}$ of  Lemma~\ref{lem:scarlatti4}.
Firstly, we have   $[E:\Q] \leq  [k:\Q] [F((\cdot/\beta)):\Q][F(\psi):\Q]
\leq [k:\Q] 2 q.$ Then, using Proposition~\ref{prop: upper bound on disc} with 
$L_1=k,L_2=F((\cdot/\beta)) F(\psi)$ we obtain $$|\textrm{Disc}(E/\Q)| \leq  |\textrm{Disc}(k/\Q)|^{[L_2:\Q]}
|\textrm{Disc}(L_2/\Q)|^{[k:\Q]} \leq  |\textrm{Disc}(k/\Q)|^{2q}
(|\beta|q^q)^{[k:\Q]}.$$ Since  $N>3$ and $q\leq (\log x)^{9/10},|\beta|\leq (\log x)^{N/3}$ it is easy to verify that 
$$|\textrm{Disc}(k/\Q)|^{2q \gamma_1} \leq x^{1/3},
(|\beta|q^q)^{\gamma_1 [k:\Q]}\leq x^{1/3},( [k:\Q] 2 q )^{\gamma_1 [k:\Q] 2 q } \leq x^{1/3},$$
 hence $ (|\mathrm{Disc}(E/\Q)| [E:\Q]^{[E:\Q]})^{\gamma_1}\leq |\textrm{Disc}(k/\Q)|^{2q \gamma_1}(|\beta|q^q)^{\gamma_1 [k:\Q]}
( [k:\Q] 2 q )^{\gamma_1 [k:\Q] 2 q } \leq x $. \end{proof}

 \subsection{Not perfectly unstable fields}

\begin{theorem} \label{Thm: infinitely many with delta=1}
There are infinitely many number fields $L$ with $\delta_L=1$ and $[L:\Q]=6$. 
\begin{proof}Define 
 $G:=\mathbb{F}_4 \rtimes \mathbb{Z}/3\mathbb{Z} \simeq_{\text{gr.}} \mathbb{F}_4 \rtimes \mathbb{F}_4^{\times} \simeq_{\text{gr.}} A_4$ ,
where the action is given by multiplication by the third  root of unity on $\mathbb{F}_4$. 
It  acts on the vertices of the $3$-dimensional cube $\{\pm 1\}^3$ by isometries. This induces an action on the set $X$ consisting of the $6$ faces of the cube, that we describe as follows: Write  $X:=\{x_1, \ldots, x_6\}$ and  consider the group of permutations that preserve the decomposition
$$X:=\{x_1,x_4\} \cup \{x_2,x_5\} \cup \{x_3,x_6\}.$$ The group contains   the element  $\rho:=(x_1 \mapsto x_2 \mapsto x_3)(x_4 \mapsto x_5 \mapsto x_6)$ 
of order $3$ and  two commuting involutions
$ \sigma_1:=(x_1 \mapsto x_4)(x_2 \mapsto x_5), \sigma_2:=(x_2 \mapsto x_5)(x_3 \mapsto x_6).$
These elements generate  $G$ and give us an explicit realization of the above action. 

Let $G_{\Q}$ denote the absolute Galois group.
To realize the previous action  as a $G_{\Q}$-set in infinitely many different ways, it 
suffices to start with a cyclic cubic extension $E/\Q$, say  $E:=\Q(\text{cos}(\frac{2\pi}{7}))$, whose class number is $1$.
Now  take any element $\alpha\in E^{*}$ with  $\alpha \notin E^{*2}$
and whose norm down to $\Q$ is a square. An explicit choice of $\alpha$ is as follows:
pick any prime $p$ congruent to $1$ modulo $7$ and let  $\pi_1$ be a generator of one of the three prime ideals above $p$. 
Then   $E(\sqrt{\alpha})/\Q$ gives the desired sextic extension, where
 $\alpha:=\pi_1\sigma(\pi_1)$  and  $\sigma$ is a generator of $\text{Gal}(E/\Q)$. As $p$ varies we get infinitely many different extensions.   
\end{proof}
\end{theorem}
To exemplify concretely the construction at the end of the proof above, the extension $\Q(\text{cos}(\frac{2\pi}{7}))$ can be given by 
an element $\alpha$ with minimal polynomial
 $\alpha^3+\alpha^2-2\alpha-1,$
hence, $\alpha$ is a unit in the ring of integer with norm equal to $1$. It is easy to see that  the polynomial
$x^6+x^4-2x^2-1$ gives the desired Galois set. 

\section{Proofs of Theorems~\ref{thm:therm1},~\ref{thm:perfectly0} and~\ref{thm:perfectly}}
\subsection{Proof of Theorem~\ref{thm:therm1}} The proof is an application of 
 Theorem~\ref{thm:analytic} with  $  \c P$ being the complement of $S_0(L/\Q)$  which is  defined in~\eqref{def:nationwide}.
By Proposition~\ref{prop:tescoexpress} the primes   $p\in S_0(L/\Q)$ automatically give a local rational point at the places above $p$.
To verify that $\c P$ satisfies the first two properties of  Definition~\ref{def:indepe} 
 we use Lemma~\ref{lem:scarlatti5} with $\c A=\c P$, $k=N(L)$, $N=600$ and $S$ being the 
set of permutations in $\textrm{Gal}(N(L)/\Q)$  that in their cycle decomposition have at least one odd cycle. Let $\c L(t):=\exp(t^{10/9})$ so that 
  if $x\geq \max\{\c L(q_\psi),\exp(\beta^{1/200})\}$  then $q\leq (\log x)^{9/10},|\beta| \leq (\log x)^{200}$, therefore, the asymptotic 
in  Lemma~\ref{lem:scarlatti5} implies the required second property in Definition~\ref{def:indepe}. 
The constant $\varpi$ of Theorem~\ref{thm:analytic}  equals 
$m(1,k,S)$, which, by Chebotarev's theorem, equals $\delta_{L/\Q}$ as    defined in~\eqref{Maramene, haggio visto 'na cosa}.
Since the identity element is always in $S$ we see that $\delta_{L/\Q}\neq 0$, hence, the third condition of  Definition~\ref{def:indepe}  is also met.
The fourth and fifth property of Definition~\ref{def:indepe}
are verified by using Lemma~\ref{lem:scarlatti3} and Lemma~\ref{lem:scarlatti2} respectively. 

Finally,  we need to determine  the quantity $z_B$ defined before Theorem~\ref{thm:analytic}.
If $\delta_{L/\Q} = 1$  then $z_B=\log \log B$. Otherwise,  we have 
$$(\log \c L(\mathrm e^{3z} ))^{\frac{8}{1-\delta_{L/\Q}}} =
\exp\Big( \frac{80 z}{3( 1-\delta_{L/\Q} )}\Big), $$ which is at most $\log B$ equivalently when $z\leq \gamma' \log \log B$ for some constant 
$ \gamma'= \gamma'(L)$. Therefore, in both cases we take  $z=\gamma_L\log \log B $  for some $\gamma_L$.  
Replacing $\gamma_L$ by $\min \{\gamma_L,3/10\}$ we see that 
$$\c L(\mathrm e^{3z_B}) =\exp(\mathrm e^{10z_B/3} )
=  \exp( (\log B)^{10\gamma_L/3} ) \leq B ,$$ hence, the assumption  $\c L(\mathrm e^{3z_B})  \leq B$ of Theorem~\ref{thm:analytic} is met. 
Finally, the error term supplied exhibits the saving $ (\log \min\{\log \log B, z_B\} )^{-A} \ll_L (\log  \log \log B  )^{-A} $.
 \qed

\subsection{Proof         of Theorem~\ref{thm:perfectly0}}
 The first two parts of 
Theorem~\ref{Thm: charact genus 0 stability}
show that $\#\c A_L=1$ equivalently when every $C_{s,t}$ with a point in $L$ also has a point in $\Q$.
If $r_L=1$ then $\delta_{L/\Q}=1$ by Theorem~\ref{thm:therm1}  and~\ref{thm:analytic} with $\c P$ being the set of all primes.
Hence, the complement of $\c A_L$ is finite and by
Remark~\ref{rem:Leonardo_Leo_is_a_god} it must contain at most one prime. The proof of the first part   
 of Theorem~\ref{thm:perfectly0} 
concludes by using the first and third parts of 
Theorem~\ref{Thm: charact genus 0 stability}.  In light of the first part    of Theorem~\ref{thm:perfectly0},
the second part is the same as   $ r_L<\infty  $ being equivalent to $ \#\c A_L<\infty  $. This is the special 
case $F=\Q$ of Proposition~\ref{Prop: S0 finite iff delta=1}.  The third claim of   of Theorem~\ref{thm:perfectly0}
is deduced from the first two claims.

\subsection{Proof of the first part of Theorem~\ref{thm:perfectly}} 
This is proved in the second part of Corollary~\ref{cor:lastoneok?}.
\subsection{Proof of the second part of Theorem~\ref{thm:perfectly}} 
 It was shown in the proof of 
Theorem \ref{Thm: infinitely many with delta=1} that for any  cyclic cubic number field $E/\Q$ and  
any   $\alpha\in E^{\times}$ that is not in $E^{\times 2}$   whose norm down to $\Q$ is a square, 
 the degree $6$ number field $L_{\alpha}:=E(\sqrt{\alpha})$ always satisfies $\delta_{L_{\alpha}/\Q}=1$. We will find the desired number fields in this pool. 

We begin fixing $E:=\Q(\zeta_9+\zeta_9^{-1})$. This cyclic cubic extension of $\Q$ has class number $1$. We denote by $\sigma$ a generator of $\text{Gal}(E/\Q)$. For each $\beta\in E$ the element $\alpha:=\beta  \sigma(\beta)$ satisfies 
$$N_{E/\Q}(\alpha)=N_{E/\Q}(\beta) N_{E/\Q}(\sigma(\beta))=N_{E/\Q}(\beta)^2.
$$Hence, as long as a so constructed $\alpha$ is in $E^{\times}\setminus E^{\times 2}$
we infer  that $\delta_{L_{\alpha}/\Q}=1$. We shall now focus on making sure that $\#S_0(L_{\alpha}/\Q)  \geq 2$, which, in view of Theorem \ref{Thm: charact genus 0 stability}, will enforce $1<r_{L_{\alpha}}<\infty$.

Let us denote by $H$ the ray class field of $E$ of modulus containing all the infinite places and $8$. We will henceforth 
work   with primes of $\Q$ that split completely in $H/\Q$. By construction, for such a prime $p$, each prime above $p$ admits a generator $\pi$ such that
 $\iota(\pi)>0$
for each of real embedding $\iota: E \to \R$ and 
 $\pi \equiv 1 \pmod{8 \mathcal{O}_E}.$
For each prime $p$ of $\Q$ that splits completely in $H/\Q$, we make a choice of a prime above and a choice of such a generator
and denote by $\pi(p)$ the resulting element of $E^{\times}$. 

The final construction will be provided by an element of the form
$$\beta:=\sigma(\pi(p_1)) \sigma(\pi(p_2))  \sigma(\pi(p_3)),$$ 
where we will make sure that the decomposition groups at $p_1$ and $p_2$   have all orbits of length $2$ on the Galois set corresponding to $L_{\alpha}/\Q$. We give below the corresponding quadratic  symbol conditions, which will also clarify  the need of using $3$ primes: a somewhat 
more involved argument using the large sieve for number fields 
would allow to use two primes only.

We      claim that  there are infinitely many triples $(p_1,p_2,p_3)$ of primes that split in $H/\Q$  and  
\beq{eq:claim}{
\left(\frac{\sigma(\pi(p_1))\sigma^2(\pi(p_1))\sigma(\pi(p_2))\sigma^2(\pi(p_2))\pi(p_3)\sigma(\pi(p_3))}{\pi(p_j)} \right)=-1  
\ \textrm{for \ } j=1,2.}To see why we fix any     $p_1,p_2$ that split completely in $H/\Q$   and   prove that there are infinitely many admissible 
$p_3$ that split completely in $H/\Q$. To ease the  notation we   denote $\pi_i:=\pi(p_i)$. Let us show that
$$\left(\frac{\sigma(\pi_1)\sigma^2(\pi_1)\sigma(\pi_2)\sigma^2(\pi_2)\pi_3\sigma(\pi_3)}{\pi_1} \right) =\left(\frac{\pi_1}{\sigma(\pi_1)\sigma^2(\pi_1)\sigma(\pi_2)\sigma^2(\pi_2)\pi_3\sigma(\pi_3)} \right).
$$
To see this we apply Hilbert reciprocity to $(\sigma(\pi_1)\sigma^2(\pi_1)\sigma(\pi_2)\sigma^2(\pi_2)\pi_3\sigma(\pi_3), \pi_1),$
hence, by construction, as the $\pi_i$ are positive at all real embeddings, this symbol vanishes locally at all real places. Likewise, locally at the place above $2$ (which is inert in $E/\Q$)   the symbol is $1$, as we have ensured  that both entries are    $1$ modulo $8$. The only odd places  to be checked are 
those  above $\{p_1,p_2,p_3\}$. They yield precisely the desired result in view of Proposition \ref{HS tame case+a little bit of wild}. This shows that the resulting equality of the Legendre symbol with its swapped version holds, as desired. The same argument also shows that
$$\left(\frac{\sigma(\pi_1)\sigma^2(\pi_1)\sigma(\pi_2)\sigma^2(\pi_2)\pi_3\sigma(\pi_3)}{\pi_2} \right)=
\left(\frac{\pi_1}{\sigma(\pi_2)\sigma^2(\pi_1)\sigma(\pi_2)\sigma^2(\pi_2)\pi_3\sigma(\pi_3)} \right).$$
Thus, we can rewrite~\eqref{eq:claim} as $$-\left(\frac{\pi_1}{\sigma(\pi_1)\sigma^2(\pi_1)\sigma(\pi_2)\sigma^2(\pi_2)} \right)=
\left(\frac{\pi_j}{\pi_3\sigma(\pi_3)} \right) \ \textrm{for \ } j=1,2.$$Crucially the left-hand side is fixed as it involves only the primes above $\{p_1,p_2\}$

For the  right-hand side we note that 
$$\left(\frac{\pi_j}{\pi_3\sigma(\pi_3)} \right)=\left(\frac{\pi_j}{\pi_3} \right)  \left(\frac{\sigma^2(\pi_j)}{\pi_3} \right)=
\left(\frac{\pi_j \sigma(\pi_1)}{\pi_3} \right).
$$ The field
 $H(\sqrt{\pi_1   \sigma(\pi_1)}, \sqrt{\sigma(\pi_1) \sigma^2(\pi_1)},\sqrt{\pi_2  \sigma(\pi_2)}, \sqrt{\sigma(\pi_2) \sigma^2(\pi_2)})/E
 $
is the compositum of the   linearly disjoint extensions
 $H/E,  N(L_{\pi_1  \sigma(\pi_1)})/E 
$  and $N(L_{\pi_2  \sigma(\pi_2)})/E$.
By   Chebotarev's density theorem we can find infinitely many   $p_3$ splitting in $H/\Q$ and with any of the $6$ possible $\sigma$-orbits in 
$$\text{Gal}(N(L_{\pi_1  \sigma(\pi_1)})/E) \times \text{Gal}(N(L_{\pi_2  \sigma(\pi_2)})/E).
$$Choosing $\pi$ from the set  $\{\pi_3, \sigma(\pi_3), \sigma^2(\pi_3)\}$ we see that   the two symbols 
$$\left(\left(\frac{\pi_1   \sigma(\pi_1)}{\pi} \right),\left(\frac{\pi_2   \sigma(\pi_2)}{\pi} \right)\right)
$$
can take each of the $4$ possible values. Therefore we  can find $\pi$ such that
$$-\left(\frac{\pi_1}{\sigma(\pi_1)\sigma^2(\pi_1)\sigma(\pi_2)\sigma^2(\pi_2)} \right)=\left(\frac{\pi_j}{\pi\sigma(\pi)} \right)
\ \textrm{for \ } j=1,2.$$
For notational convenience we rename the   choice $\pi_3$ to be equal to $\pi$ itself.

We will now apply the claim above with the choice of 
$$\alpha:=\sigma(\pi_1)\sigma^2(\pi_1)\sigma(\pi_2)\sigma^2(\pi_2)\pi_3\sigma(\pi_3),
$$
given~\eqref{eq:claim}. 
By construction, the decomposition groups of $p_1$ and $p_2$ in $\text{Gal}(N(L_{\alpha})/\Q)$ consist of the
 subgroup $\text{Gal}(N(L_{\alpha})/E)$. Indeed as both the ramification index and the residue field degree are at least $4$, we see that the decomposition group is of size at least $4$. On the other hand as $p_1,p_2$ split completely in $E/\Q$,   the decomposition groups land 
in $\text{Gal}(N(L_{\alpha})/E)$, which has size $4$.  

Next, the subgroup $\text{Gal}(N(L_{\alpha})/E)$ is isomorphic to the Klein group with $4$ elements and acts on the 
corresponding Galois set $X_{L_{\alpha}}$ of $6$ elements  so  that its orbits are $3$ each of length $2$. 
Representing $X_{L_{\alpha}}$ as the $6$ faces of a cube, these $3$ orbits are precisely the $3$ pairs of opposite faces. 

Hence, at both  $p_1$ and $p_2$
the decomposition groups have only orbits of even length, namely, equal to $2$. In other words, 
 for each such $\alpha$ we have proved 
$\{p_1,p_2\} \subseteq S_0(L_{\alpha}/\Q)$. 
Since we know that 
 $\delta_{L_{\alpha}/\Q}=1$,  we conclude by  Theorem \ref{thm:therm1}
and~\ref{Thm: charact genus 0 stability} 
  that $1< r_{L_{\alpha}}<\infty$.  As we vary the triple $(p_1,p_2,p_3)$ 
 we obtain infinitely many different such extensions.

\end{document}